\documentclass[a4paper,11pt]{article}
\usepackage[latin1]{inputenc}
\usepackage{amsfonts}
\usepackage{amsthm}
\usepackage{amsmath}
\usepackage{amsfonts}
\usepackage{latexsym}
\usepackage{amssymb}
\usepackage{dsfont}
\usepackage[usenames]{color}
\usepackage{algorithmic}
\usepackage[ruled,section]{algorithm}

\usepackage{color}

\newtheorem{teo}{Theorem}[section]
\newtheorem*{teo*}{teo}
\newtheorem{lem}[teo]{Lemma}

\newtheorem{pro}[teo]{Proposition}

\newtheorem{nota}[teo]{Notation}

\theoremstyle{definition}
\newtheorem{fed}[teo]{Definition}
\theoremstyle{remark}

\newtheorem{rem}[teo]{Remark}

\def\coma{\, , \, }

\def\py{\peso{and}}
\newcommand{\peso}[1]{ \quad \text{ #1 } \quad }

\def\n0{n_{ \text{\rm \tiny o}}}

\def\noi{\noindent}

\def\EOE{\hfill $\triangle$}

\def\bm{\left[\begin{array}}
\def\em{\end{array}\right]}
\def\ben{\begin{enumerate}}
\def\een{\end{enumerate}}
\def\bit{\begin{itemize}}
\def\eit{\end{itemize}}
\def\barr{\begin{array}}
\def\earr{\end{array}}

\def\R{\mathbb{R}}
\def\C{\mathbb{C}}
\def\K{\mathbb{K}}

\def\cH{\mathcal{H}}
\def\cK{\mathcal{K}}

\def\cP{\mathcal{P}}

\def\cR{{\cal R}}
\def\cS{{\cal S}}
\def\cT{{\cal T}}

\def\cV{{\cal V}}
\def\cU{{\cal U}}
\def\cW{{\cal W}}

\def\cX{\mathcal{X}}

\def\beq{\begin{equation}}
\def\eeq{\end{equation}}
\def\pausa{\medskip\noi}

\begin{document}
\title{Dominant subspace and low-rank approximations from block Krylov subspaces without a prescribed gap}
\author{Pedro Massey
 \thanks{Partially supported by 
CONICET (112202101 00954CO) and UNLP (11X974). e-mail:{massey@mate.unlp.edu.ar}}}
\date{CMaLP, FCE-UNLP and IAM-CONICET, Buenos Aires, Argentina }
\maketitle

\begin{abstract}
We develop a novel convergence analysis of the classical deterministic block Krylov methods 
for the approximation of 
$h$-dimensional dominant subspaces and low-rank approximations of matrices
$ A\in\mathbb K^{m\times n}$ (where $\mathbb K=\mathbb R$ or $\mathbb C)$ 
in the case that there is no singular gap at the index $h$ i.e., if $\sigma_h=\sigma_{h+1}$ (where $\sigma_1\geq \ldots\geq \sigma_p\geq 0$ denote the singular values of $ A$, and $p=\min\{m,n\}$). 
Indeed, starting with a (deterministic) matrix $ X\in\mathbb K^{n\times r}$ with $r\geq h$ satisfying a compatibility assumption with some $h$-dimensional right dominant subspace of $A$, we show that block Krylov methods produce arbitrarily good approximations for both problems mentioned above. Our approach is based on recent work by Drineas, Ipsen, Kontopoulou and Magdon-Ismail on the approximation of structural left dominant subspaces. The main difference between our work and previous work on this topic is that instead of exploiting a singular gap at the prescribed index $h$ (which is zero in this case) we exploit the nearest existing singular gaps. 
\end{abstract}

\bigskip


\medskip

 

\section{Introduction}\label{sec intro}

Low-rank matrix approximation  is a central problem in numerical linear algebra (see \cite{Saad11}). It is well known that truncated singular value decompositions (SVDs) of a matrix $ A\in\K^{m\times n}$ 
(for $\K=\R$ or $\K=\C$) produce optimal solutions to this problem (\cite{Bhatia,GVL13,HJ91,Saad11}). Indeed, let $ A= U \Sigma V^*$ be a SVD and let 
$\sigma_1\geq \ldots\geq \sigma_p\geq 0$ be the singular values of $ A$, where $p=\min\{m,n\}$. 
Given $1\leq h\leq \text{rank}( A)$, recall that 
a truncated SVD of $ A$ is given by $ A_h= U_h \Sigma_h V_h^*$, 
where the columns of $ U_h$ and $ V_h$ are the top $h$ columns of
$ U$ and $ V$, respectively, and $ \Sigma_h$ is the diagonal matrix with main diagonal given by $\sigma_1,\ldots,\sigma_h$. In this case, we have that
$\| A- A_h\|_{2,F}\leq \| A- B\|_{2,F}$ for every $ B\in\K^{m\times n}$ with $\text{rank}( B)\leq h$, where $\|\cdot\|_{2,F}$ stand for spectral and Frobenius norms respectively. Nevertheless, it is well known that (in general) computation of a SVD of a matrix is expensive. This motivates the efficient numerical computation of approximations of truncated SVDs of matrices \cite{D19,Gu15,HMT11,MM15,Saad11,WZZ15,Wood14}.

A closer look at the optimal approximations shows that 
they can be described as $ A_h= P_h A$, where $ P_h\in\K^{m\times m}$ is the orthogonal projection onto the range $R(U_h)=\mathcal U_h$, spanned by the top $h$ columns of $ U$. Hence, one of the main strategies for computing low-rank approximations is the computation of convenient $h$-dimensional subspaces $\cT\subset\K^m$ and considering 
the corresponding low rank approximations. 
There are several methods for the efficient computation of low-rank approximations of the form $PA$ for an orthogonal projection $P\in\K^{m\times m}$, based on the construction of convenient $h$-dimensional subspaces $R(P)=\cT$ (equivalently, orthonormal sets of $h$ vectors).  Among others, implementations of the subspace iteration (or block power) and block Krylov methods have become very popular \cite{D19,GVL13,Gu15,MM15,Saad11,Sai19,WZZ15}. Yet, even if $PA$ is a good low-rank approximation of $A$ (that is, if $\|PA-A\|_2\approx \sigma_{h+1}$), $PA$ and $A$ might not share some other features. For example, the singular values $\sigma_i(PA)$ and $\sigma_i(A)$ might be quite different for some $1\leq i\leq h$. It could also be the case that 
the range $R(P A)\subset \K^m$ is not close to any $h$-dimensional left dominant subspace $\cU_h\subset \K^m$ of $A$; here, the distance between subspaces is measured in terms of the principal angles between them (see \cite{MM15}). Thus, to derive low-rank approximations that share some other features with $A$, it seems natural to consider first the construction of subspaces $\cT$ that are close to the subspaces $\cU_h$. Once this is achieved, these subspaces can be used to construct approximated truncated SVDs of the form $P\,A$ that do share some other features with $A$. Moreover, the subspaces $\cT$ themselves are also relevant in the study of principal component analysis \cite{Jolli86} (since they are approximations of the principal components $\cU_h$) and for the
estimation of the leverage scores for the construction of CUR decompositions 
(see the recent survey \cite{DPDM23} and the references therein). 

As opposed to the low-rank approximation problem of the matrix $A$ (see \cite{DI19}), the singular gap $\sigma_h-\sigma_{h+1}\geq 0$ plays a role in the approximation of the $h$-dimensional left dominant subspace $\cU_h$ of $A$. Indeed, if the singular gap is positive then the  condition number of computing the (uniquely determined) left dominant subspace $\cU_h$ of $A$ depends (up to small factors) on the inverse of the singular gap $(\sigma_h-\sigma_{h+1})^{-1}$, irrespectively of the method being used to approximate $\cU_h$ (see \cite{Sun96,Vann22}). This fact is reflected in the convergence analysis of deterministic iterative methods used to compute approximations 
of $\cU_h$, as those considered in \cite{D19,Sai19,WZZ15}; indeed, in the deterministic setting, the upper bounds for the principal angles between $\cU_h$ and its approximations obtained in the previous works become arbitrarily large as the singular gap $\sigma_h-\sigma_{h+1}>0$ tends to $0$. 
Furthermore, in case the singular gap is zero then we get an infinite family of $h$-dimensional 
left dominant subspaces of $A$; this introduces a new problem, 
that is {\it choosing} a convenient $h$-dimensional left dominant subspace of $A$ to approximate.

The previous facts could lead us to believe that the problem of approximating $h$-dimensional left dominant subspaces of $A$ when the singular gap $\sigma_h-\sigma_{h+1}=0$ is not well posed for its analysis in the deterministic setting. In this manuscript, we challenge this belief for the block Krylov iterative method. Indeed, in case $\sigma_h=\sigma_{h+1}$ we show that the existence of infinitely many $h$-dimensional left dominant subspaces of $A$ can be used to our advantage, by developing a method for choosing a convenient such dominant subspace that is close to the subspace obtained by the block Krylov method. To do this we adapt some of the main ideas of \cite{D19}, to deal with the approximation of left dominant subspaces. Indeed, we consider a starting guess matrix $ X\in\K^{n\times r}$ that satisfies some compatibility assumptions with $ A$, which can always be achieved even with $r=h$ (i.e., for a minimal choice of $r$).
Our approach is based on enclosing $\sigma_j>\sigma_{j+1}=\sigma_h=\sigma_k>\sigma_{k+1}$ in such a way that $j<h$ and $k\geq h$ are the nearest indices  for which there are singular gaps. 
These gaps appear explicitly in the upper bounds related to our convergence analysis of block Krylov methods.  In this context, we show that block Krylov subspaces produce arbitrarily good $h$-dimensional approximations of conveniently chosen left and right $h$-dominant subspaces. Moreover, we show that block Krylov spaces can also be used to compute 
 arbitrarily good approximations of $A$ of rank at most $h$ that share some other features with $A$, even when the singular gap $\sigma_h-\sigma_{h+1}=0$ 
(see Section \ref{sec dom subs4} for a detailed description of the problems mentioned above).  
 Thus, our results complement the convergence analysis in \cite{D19}.
Besides the fact that the no singular gap
case is of interest (due to the common occurrence of repeated singular values in applications with some degree of symmetry), we believe that our 
approach to the convergence analysis of block Krylov iterative methods can be extended in different directions, not only to cover the convergence analysis of other methods but also to allow for a more general understanding of subspace approximation problems (e.g. for clustered singular values).

The paper is organized as follows. In Section \ref{sec new prelis} we recall the notions of principal angles between subspaces, dominant subspaces and their relation to SVDs and 
we describe the context and main problems considered in this work. 
In Section \ref{sec main results} we state our main results on $h$-dimensional dominant subspace approximations (Section \ref{sec main probs uno}) and low-rank approximations by matrices of rank $h$ (Section \ref{sec low rank approx} ) when there is no singular gap at the index $h$. At the end of Section, we include some discussion
on how is that our approach can be adapted to deal with more general situations (such as the clustered singular value case). In Section \ref{sec proofs} we present the proofs of the results described in Section \ref{sec main results}; some of these proofs require some technical facts that we consider in Section \ref{apendixity} (Appendix). 

\section{Preliminaries and description of the main context}\label{sec new prelis}

We begin by recalling some geometric notions that play a central role
in the convergence analysis of iterative algorithms. Then, we describe 
the context and problems that are the main motivation of our work.

\subsection{Principal angles between subspaces}\label{sec aux angles}

 Let $\cS,\,\cT\subset \K^n$ be two subspaces of dimensions $s$ and $t$ respectively. Let $ S\in \K^{n\times s}$ and $ T\in\K^{n\times t}$ be isometries (i.e., matrices with orthonormal columns) such that $R( S)=\cS$ and $R( T)=\cT$. Following \cite{GVL13}, we define the principal angles between $\cS$ and $\cT$, denoted 
$$0\leq \theta_1(\cS,\cT)\leq \ldots\leq \theta_k(\cS,\cT)\leq \frac{\pi}{2} \peso{where} k=\min\{s,\,t\}\,, $$
determined by the identities $\cos(\theta_{i}(\cS,\cT))=\sigma_i( S^* T)=\sigma_i( T^* S)$, for $1\leq i\leq k$; in this case the roles of $ S$ and $ T$ are symmetric.
If we assume that $s\leq t$ (so $k=s$) the principal angles can be also determined in terms of the identities 
\begin{equation}\label{eq sobre sen ang princ1}
\sin(\theta_{s-i+1}(\cS,\cT))=\sigma_i( ( I-  T T^*) S)=\sigma_i( ( I-  T T^*) S  S^*)=\sigma_i(( I- P_\cT) P_\cS)
\end{equation}
for $1\leq i\leq s$,  where $ P_\cH\in\K^{n\times n}$ denotes the orthogonal projection onto a subspace $\cH\subset \K^n$; it is worth noticing that in this last case the roles of $ S$ and $ T$ (equivalently the roles of $ P_\cS$ and $ P_\cT$) are not symmetric (unless $s=t$). 
Principal angles can be considered as a vector valued measure of the distance between the subspaces $\cS$ and $\cT$. Following \cite{SS90} we let $ \Theta(\cS,\cT)=\text{diag}(\theta_1(\cS,\cT),\ldots,\theta_s(\cS,\cT))$ denote the diagonal matrix with the principal angles in its main diagonal. As a consequence of the previous facts we get that 
$$\|\sin  \Theta(\cS,\cT)\|_{2,F}=\|( I- P_\cT) P_\cS\|_{2,F}
$$ are (scalar measures of) the {\it angular distances} between $\cS$ and $\cT$ (see 
\cite{GVL13,SS90}).

\subsection{Dominant subspaces }\label{sec dom subs2} 

We begin with a formal description of the class of dominant subspaces of a matrix, without assuming a singular gap. 
Let $A\in\K^{m\times n}$ and let $\sigma_1\geq \ldots  \geq \sigma_p\geq 0$, where $p=\min\{m,\,n\}$, denote its singular values. Let $\cS'\subset \K^m$ be a subspace of dimension $1\leq h\leq \text{rank}( A)\leq p$. We say that $\cS'$ is a {\it left dominant subspace} for 
$ A$ if $\cS'$ admits an orthonormal basis $\{ w_1,\ldots, w_h\}$ such that 
$ A  A ^* w_i=\sigma_i^2\,  w_i$, for $1\leq i\leq h$.
Equivalently, $\cS'$ is a left dominant subspace for 
$ A$ if the $h$ largest singular values of $  P_{\cS'}   A$ are $\sigma_1\geq  \ldots\geq \sigma_h$. In this case we have that 
$$
\| P_{\cS'}  A -   A\|\leq \| Q  A -  A\|
$$ for every orthogonal projection $ Q\in\K^{m\times m}$ with rank$( Q)=h$ and every unitarily invariant norm; that is, $ P_{\cS'} A$ is an optimal low-rank approximation of $ A$ (see \cite[Section IV.3]{Bhatia}).

On the other hand, we say that $\cS\subset \K^n$ is a {\it right dominant subspace} for $ A$ if
$\cS$ admits an orthonormal basis $\{ z_1,\ldots, z_h\}$ such that 
$ A^*  A z_i=\sigma_i^2 z_i$, for $1\leq i\leq h$.
 Similar remarks apply also to right dominant subspaces. It is interesting to notice that 
the class of $h$-dimensional left dominant subspaces of $ A$ coincides with
the class of $h$-dimensional  right dominant subspaces of $ A^*$; in what follows we will make use of this fact.

\subsection{Dominant subspaces and SVDs}\label{sec dom subs3} 

  Let $ A=  U  \Sigma  V^*$ be a full SVD for $ A\in\K^{m\times n}$, where $\K=\R$ or $\K=\C$, 
$ \Sigma\in \R^{m\times n}$ and $ U\in\K^{m\times m}$ and $ V\in\K^{n\times n}$ are unitary (orthogonal when $\K=\R$) matrices. In this case $\Sigma$ is a (rectangular) diagonal matrix, with diagonal entries given by the singular values of $A$. In what follows we let $  u_j$ (respectively $  v_j$) denote the columns of $ U$ (respectively of $ V$).

Given $1\leq h\leq m$, we define the subspace 
$\cU_h=\text{Span}\{ u_1,\ldots, u_h\}\subset \K^m$; similarly, if $1\leq h\leq n$, we let 
$\cV_h=\text{Span}\{ v_1,\ldots, v_h\}\subset \K^n$. Then, $\cU_h$ and $\cV_h$ are left and right dominant subspaces respectively. 
In case $\sigma_h>\sigma_{h+1}$ then it is well known that the left (respectively right) dominant subspace for $ A$ of dimension $h$ is uniquely determined; hence, in this case $\cU_h$ and $\cV_h$ do not depend on our particular choice of SVD for $ A$. 

On the other hand, if $\sigma_h=\sigma_{h+1}$ then we have a continuum class of 
$h$-dimensional left dominant subspaces: indeed, let 
$0\leq j=j(h)< h<k=k(h)$ be given by
$j(h)=\max\{ 0\leq \ell < h\ : \ \sigma_\ell> \sigma_h\}$, where we set $\sigma_0=+\infty$ and 
$k=k(h)=\max\{ 1\leq \ell\leq \text{rank}( A) \, : \, \sigma_\ell=\sigma_h\}$. 
If we further let $\cU_0=\{0\}$ then, it is straightforward to check that an $h$-dimensional subspace $\cS'$ is a left dominant subspace for $ A$ if and only if there exists an $(h-j)$-dimensional subspace $\cU\subset \cU_k\ominus \cU_j:=\cU_k\cap \cU_j^\perp\subset \K^m$ such that $\cS'=\cU_j\oplus \cU$.
Therefore, we have a natural parametrization of $h$-dimensional left dominant subspaces in terms of 
subspaces $\cU$ that vary over
the Grassmann manifold  of $(h-j)$-dimensional subspaces of $\cU_k\ominus \cU_j\subset \K^m$.

It is a basic fact in linear algebra that given $\cS'$ a left dominant subspace of dimension $h\geq 1$,  there exists a SVD, $ A=  U  \Sigma  V^*$ such that $\cS'=\cU_h$ i.e., the subspace spanned by the top $h$ columns of $ U$; and a similar fact also holds for 
right dominant subspaces.

\subsection{Main problems considered in this work}\label{sec dom subs4} 

Throughout the rest of the paper we consider the following block Krylov algorithm for approximation of left dominant subspaces and computation of low-rank approximations (see \cite{D19}).

\smallskip

\begin{algorithm}
\caption{(Block Krylov algorithm for left dominant subspace and low-rank approximation)}\label{algoalgo}
\centerline{
}
\begin{algorithmic}[1]
\REQUIRE  $A\in\K^{m\times n}$, starting guess $ X\in \K^{n\times r}$; target rank $h\leq \text{rank}(A)$; power $\ell\geq 0$.\\
$\quad~$ Set 
\begin{equation} \label{eq defi Krylov2}
 K_\ell=K_\ell(A,X):=( A X\quad  ( A A^*) A X
\quad \cdots  \quad  ( A A^*)^\ell A X)\in\K^{m\times (\ell+1)\cdot r}
\end{equation}
$$ \cK_\ell=\cK_\ell(A,X):=R(K_\ell)\subset \mathbb K^m\,.$$
$\quad~$  Test that $d:=\dim \cK_\ell\geq h$. In this case:
\ENSURE $\hat{ U_h}\in\K^{m\times h}$ with orthonormal columns
\STATE Compute an orthonormal basis $ U_K\in\K^{m\times d}$ for $\cK_\ell$.
\STATE Set $ W= U_K^*  A\in \K^{d\times n}$ (notice that rank$( W)\geq h$).
\STATE Compute $ U_{W,h}\in \K^{d\times h}$ isometry, such that $R( U_{W,h})$ is a left dominant subspace of $ W$.
\STATE Return: $U_K\in \mathbb K^{m\times d}$ and $\hat { U}_h= U_K\, U_{W,h}\in \K^{m\times h}$.
\end{algorithmic}
\end{algorithm}

\smallskip

Once Algorithm \ref{algoalgo} is performed, we consider the output matrices $U_K$ and $\hat U_h$. 
We assume further that $\sigma_h=\sigma_{h+1}$; 
 in this setting, our first main problem is to show the existence of some $h$-dimensional subspace  $\cT\subseteq \cK_\ell$ that is close to {\it some}
$h$-dimensional left dominant subspace $\cU_h$ of $A$. In this context, proximity between subspaces is measured by $\|\sin \Theta(\cU_h,\cT)\|_{2,F}$ i.e., in terms of (the spectral or Frobenius norm of) the sines of the principal angles between the subspaces $\cU_h$ and $\cT$ (see Section \ref{sec main probs uno}). Once we establish the existence of $\cT\subseteq \cK_\ell$ as above, we get the low-rank approximation $P_{\cT}A$ of $A$, where $P_{\cT}$ denotes the orthogonal projection onto $\cT$. We point out that the previous approach does not provide an effective way (algorithm) to compute $\cT$.

Therefore our second main problem is to obtain upper bounds for the approximation error $\|A-\hat A_h\|_{2,F}$, for the low-rank matrix $\hat A_h=\hat U_h \hat U_h^* A$ computed in terms of the output of Algorithm \ref{algoalgo}.
Further, we require that the upper bound for the  approximation error of $A$ by $\hat A_h$ becomes arbitrarily close to $\|A-A_h\|_{2,F}$ (optimal approximation error as described at the beginning of Section \ref{sec intro}) as the power $\ell$ increases. 
 Hence, by solving this second problem, we obtain (in an effective way) the low-rank approximation $\hat A_h$ of $A$ (see Section \ref{sec low rank approx}) that behaves much like the optimal low-rank approximations $A_h$ of $A$. 

In the case that there is a singular gap i.e., $\sigma_h>\sigma_{h+1}$, these problems have been recently solved in \cite{D19}. In this work we adapt the approach considered in \cite{D19} to construct approximations of dominant spaces and low-rank approximation of $A$, based on the block Krylov subspaces $\cK_\ell$,
in the case that there is no singular gap at the index $h$. 

\section{Main results}\label{sec main results}

In this section, we state our main results related to dominant subspace approximations and low-rank matrix approximations in terms of block Krylov subspaces. The proofs of these results are considered in Section \ref{sec proofs}. At the end of this section, we include some comments and further research problems related to the present work.

Our results are motivated by the recent work of P. Drineas, I.C.F. Ipsen, E.M. Kontopoulou and M. Magdon-Ismail \cite{D19}. In that work, the authors merged a series of
techniques, tools and arguments that lead to structural results related to 
the approximation of dominant subspaces from block Krylov spaces in the presence of a singular gap. 
The convergence analysis obtained in \cite{D19} has a deep influence in our present work; indeed, we shall follow some of the lines developed in that work, that we refer to as the \textit{DIKM-I theory}. Of course, at some points, we have to depart from those arguments to deal with the no-singular-gap case.

\subsection{Approximation of dominant subspaces by block Krylov spaces}\label{sec main probs uno}

In what follows we consider the convergence analysis of the block Krylov iterative method for approximating left dominant subspaces of a matrix $A$ (Algorithm \ref{algoalgo}). 
Once the Algorithm \ref{algoalgo} is performed, we describe the output matrix $\hat { U}_h$ in terms of its columns $\hat { U}_h=(\hat { u}_1,\ldots,\hat { u}_h)$. We also consider the matrices $\hat { U}_i=(\hat { u}_1,\ldots,\hat { u}_i)\in\K^{m\times i}$, for $1\leq i\leq h$.

 As before, let $ A\in\K^{m\times n}$ with singular values $\sigma_1\geq \ldots\geq \sigma_p$, for $p=\min\{m,n\}$. Given $1\leq h\leq \text{rank}( A)\leq p$, we let 
$0\leq j(h)< h$ be given by
\beq\label{eq defi j}
j=j(h)=\max\{ 0\leq \ell < h\ : \ \sigma_\ell> \sigma_h\}\eeq where we set $\sigma_0=+\infty$ and 
\beq \label{eq defi k}
k=k(h)=\max\{ 1\leq \ell\leq \text{rank}( A) \, : \, \sigma_\ell=\sigma_h\}\,.\eeq
Since $h\leq \text{rank}( A)$, we get that $\sigma_k=\sigma_h>0$. 
As mentioned in the preceding sections, we will focus on the case when $h<k$ (i.e., when $\sigma_h=\sigma_{h+1}$). Moreover, we will further assume that $h\leq r<k$ so that $\sigma_h=\sigma_r=\sigma_k$) where $X\in\mathbb K^{n\times r}$ is the starting guess in Algorithm \ref{algoalgo}; otherwise (if $k\leq r$) we could simply apply the DIKM-I theory, using the singular gap $\sigma_k>\sigma_{k+1}$ (since, in the generic case, $X$ is full rank and satisfies the compatibility hypothesis of the DIKM-I theory).
Let $A=U\Sigma V^*$ be a full SVD of $A$.
In case $1\leq k< p$ then
we consider the following partitioning of $U$, $\Sigma$ and $V$ 
$$\Sigma=\begin{pmatrix}  \Sigma_k & \\ &  \Sigma_{k,\perp}\end{pmatrix}\ , \ \ 
U=\begin{pmatrix}  U_k & 
 U_{k,\perp}\end{pmatrix}\ , \ \ 
V=\begin{pmatrix}  V_k & 
V_{k,\perp}\end{pmatrix}\,.$$

Algorithm \ref{algoalgo} will provide a reasonable output as long as the starting guess matrix 
satisfies certain compatibility conditions with $A$. 
Hence, we consider the following

\begin{fed}Given $ X\in\K^{n\times r}$  we say that $( A, X)$ is {\it $h$-compatible} if there is an $h$-dimensional right dominant subspace $\cS\subset \K^n$ for $ A$, with $$ \Theta(\cS,R( X))<\frac{\pi}{2}\, I\,.$$
\EOE
\end{fed}

Given $ X\in\K^{n\times r}$  notice that $( A, X)$ is $h$-compatible if and only if 
$\dim ( X^*\cS)=h$, for some $h$-dimensional right dominant subspace $\cS$ of $A$.

\medskip

Throughout the rest of the work, we fix $1\leq h\leq \text{rank}( A)\leq p=\min\{m,n\}$ and we let  $0\leq j=j(h)<h\leq k=k(h)\leq \text{rank}(A)$ be defined as in Eqs. \eqref{eq defi j} and \eqref{eq defi k}. The next result will allow us to show
that block Krylov methods produce arbitrarily good approximations of right and left dominant subspaces under the previous hypothesis (see Remark \ref{rem strat} below). 
In what follows, given a matrix $ Z$ we let $ Z^\dagger$ denote its Moore-Penrose pseudo-inverse.

\begin{teo}\label{first main result} 
Let $t\geq 0$ and let  $\phi(x)$ be a polynomial 
of degree at most $2t+1$ with odd powers only,
such that 
$\phi(\sigma_1),\,\ldots,\,\phi(\sigma_k)>0$.
 Let $( A, X)$ be $h$-compatible and let $\cK_t=\cK_t( A, X)$. 
 Then, there exists an $h$-dimensional left dominant subspace $\cS'$ for $ A$ such that
\begin{eqnarray*}
\|\sin \Theta(\cK_t,\cS')\|_{2,F}&\leq & 2\, \|\sin \Theta(R( V_k^* X), V_k^*\cV_j)\|_{2,F} + \\ \\ & &
\|\phi( \Sigma_{k,\perp})\|_2\|\phi( \Sigma_{k})^{-1}\|_2 \|  V_{k,\perp}^*  X
( V_k^*  X) ^\dagger\|_{2,F}\,.
\end{eqnarray*}
In case $j=0$ (respectively $k=\text{rank}( A)$) the first term (respectively the second term) should be omitted in the previous upper bound.
Moreover, we have the inequality $$  \Theta(R( V_k^* X), V_k^*\cV_j)\leq  \Theta(R( X),\cV_j)\,.$$  
\end{teo}
\begin{proof}
See Section \ref{sec prueba teo 2.1}.
\end{proof}

We point out that Theorem \ref{first main result} above is related to
\cite[Theorem 2.1]{D19} from the DIKM-I theory. In case $\sigma_h=\sigma_{h+1}$ (and hence $k>h$), the hypothesis in Theorem \ref{first main result}  involves
the (continuum) class of $h$-dimensional left dominant subspaces of $A$; that is, we are allowed to consider any such dominant subspace to test our assumptions. In particular, the hypothesis is satisfied in a generic case.

As mentioned in Section \ref{sec dom subs3}, since $\sigma_j>\sigma_{j+1}$ and $\sigma_k>\sigma_{k+1}$ the
subspaces $\cV_j=R(V_j)$ and $\cV_k=R(V_k)$ do not depend on the particular SVD of $A=U\Sigma V^*$ being considered. 
Hence, the principal angles $\Theta(R( V_k^* X), V_k^*\cV_j)$ in Theorem \ref{first main result}
do not depend on the particular SVD of $A$, since they coincide with 
$\Theta(R( P_{\cV_k} X), P_{\cV_k}\cV_j)$ (here $P_{\cV_k}\in\mathbb K^{n\times n}$ denotes the orthogonal projection onto $\cV_k\subset \mathbb K^n$).  

To apply Theorem \ref{first main result} we need to bound
the principal angles $\Theta(R( V_k^* X),R(V_k^*\cV_j))$ (notice that \cite[Theorem 2.1]{D19} does not require such bound). 
We remark that the second inequality in  Theorem \ref{first main result} 
provides an alternative method to have control of the previous principal angles.

Using the results from \cite{ZKnya} we can interpret the second term in the upper bound in Theorem \ref{first main result} as a measure of how close the subspaces $R(X)$ and (the $k$-dimensional right dominant subspace) $\cV_k$ are; indeed, the smaller this second term is, the closer the subspaces are. The distance between these subspaces does not seem to have been considered before 
in the analysis of convergence of dominant subspaces obtained by block Krylov methods, under our present assumption that $h\leq r<k$ (where the dimension of the dominant subspace is strictly bigger than that of $R(X)$).

\begin{rem}[Convergence analysis to dominant subspaces]\label{rem strat} Theorem \ref{first main result} suggests a strategy to obtain 
a convergence analysis of subspaces obtained from the block Krylov method to $h$-dimensional left dominant subspaces. We now describe this strategy in detail for the benefit of the reader. With the notation of Theorem \ref{first main result} we consider the following steps:
\begin{enumerate}
\item[Step 1.] Beginning with $A$, $X$ and $q\geq 0$, we first consider the (auxiliary) block Krylov 
matrix 
$$
\tilde K_q:= K_q(A^*,AX)=(A^*(AX)\quad (A^*A) A^*(AX) \quad \cdots \quad (A^*A)^q A^*(AX))
$$
constructed in terms of $A^*\in\mathbb K^{n\times m}$ and (the auxiliary starting guess matrix) $AX\in \mathbb K^{m\times r}$. Assume that $1\leq j <h$ so that $\sigma_j(A^*)=\sigma_j(A)>\sigma_{j+1}(A)=\sigma_{j+1}(A^*)$.

Notice that if $\psi(x)\in \K[x]$ is any polynomial of degree at most $2q$ with even powers only and such that $\psi(\sigma_1)\geq \ldots\geq \psi( \sigma_j)>0$ then $\Psi_q=V(\psi(\Sigma) \cdot\Sigma^2)V^*X$ is such that $R(\Psi_q)\subset \tilde K_q$; here we have used the convention $\Sigma^{2\ell}=(\Sigma^\ell)^*\cdot \Sigma^\ell\in \K^{n\times n}$ for the sake of simplicity. Moreover,  it is well known that there are convenient choices of $\psi(x)$ which warrant that 
the angles between $R(\Psi_q)$ and (the $j$-dimensional left dominant subspace of $A^*$) $\cV_j$ become arbitrarily small as $q$ increases. 
\item[Step 2.] 
Assume that $k<\text{rank}(A)$; an application of Theorem \ref{first main result} to the matrix $A$, with starting guess matrix $\Psi_q$ and $t=0$ (so that $\phi(x)=x$), shows that there exists 
an $h$-dimensional left dominant subspace $\cS'$ of $A$ such that 
\begin{eqnarray}\label{ecuac pepi1}
\|\sin\Theta (\cK_{q+1}(A,X), \cS' )\|_{2,F}\leq 2\,\|\sin\Theta(R(\Psi_q)), \cV_j)\|_{2,F}+\frac{\sigma_{k+1}}{\sigma_k} \| V_{k,\perp}^* \Psi_q (V_k^* \Psi_q)^\dagger\|_{2,F}\,.
\end{eqnarray}
where we have used the inclusion $\cK_0(A^*, \Psi_q)\subset \cK_{q+1}(A,X)$, the properties of principal angles described in Remark \ref{rem sobre ang princ}, and that $A^*=V\Sigma U^*$ is a SVD of $A^*$.
\item[Step 3.] As mentioned in Step 1., the first term in the upper bound in Eq. \eqref{ecuac pepi1} is known to become arbitrarily small for convenient choices of $\psi(x)$; we will show that there are choices of $\psi(x)$ for which both the expressions $\|\sin\Theta(R(\Psi_q)), \cV_j)\|_{2,F}$
and $\| V_{k,\perp}^* \Psi_q (V_k^* \Psi_q)^\dagger\|_{2,F}$ become arbitrarily small as $q$ increases. 
\end{enumerate}
We point out that the cases in which $j=0$ or $k=\text{rank}(A)$ can be tackled in a similar way.
The reader will notice that Theorem \ref{first main result} allows for other possible approaches to the  convergence analysis of subspaces obtained from the block Krylov algorithm to $h$-dimensional left dominant subspaces. 
\end{rem}

In what follows we consider the different steps of the strategy for the convergence analysis
of the block Krylov subspaces to the $h$-dimensional left dominant subspaces described in Remark 
\ref{rem strat}. We begin by obtaining an upper bound related to the claim in Step 1. 
Thus, given $A$ with singular values $\sigma_1\geq\ldots\geq \sigma_p\geq 0$ (with $p=\min\{m,n\}$), we introduce
$$
\gamma_i=\frac{\sigma_i-\sigma_{i+1}}{\sigma_{i+1}}\geq 0 \peso{for} 1\leq i\leq \text{rank}(A)-1\,.
$$
The following result is a straightforward consequence of the results from \cite{D19}.
\begin{teo}\label{pro 3.5}
 Let $( A, X)$ be $h$-compatible and assume that $1\leq j<h$, so then $\gamma_j>0$.
Let $\psi(x)\in \K[x]$ be a polynomial of degree at most $2q$ with even powers only and such that $\psi(\sigma_1), \ldots,\psi(\sigma_j)>0$. If we let $\Psi_q= V(\psi(\Sigma)\cdot \Sigma^2)V^*X$ then
$$
\|\sin\Theta(R(\Psi_q),\cV_j)\|_{2,F}\leq \|\psi(\Sigma_{j,\perp})\|_2\,
\|\psi(\Sigma_{j})^{-1}\|_2\, 
 \| V_{j,\perp}^* X( V_j^*  X)^\dagger\|_{2,F}\,
\frac{1}{(1+\gamma_j)^2}
 $$
\end{teo}
\begin{proof}
See Section \ref{sec 4.3}.
\end{proof}

\medskip

In the next result we obtain an upper bound required in Step 3 of Remark \ref{rem strat}. 

\begin{teo}\label{third main result S1} 
 Let $( A, X)$ be $h$-compatible, let $q\geq 0$ and assume that $\sigma_{k+1}>0$. Let $T_{2q}(x)$ be the Chebyshev polynomial of the first kind of 
degree $2q$ and set $\psi(x)=T_{2q}(x/\sigma_{k+1})$. Then, $\psi(\sigma_1)\geq \ldots\geq \psi(\sigma_k)>1$ and if we let
$\Psi_q=V(\psi(\Sigma)\cdot \Sigma^2)V^*X$ then
\begin{equation}\label{suert1}
\|\sin\Theta(R(\Psi_q),\cV_j)\|_{2,F}\leq \frac{\psi(\sigma_{k})}{\psi(\sigma_{j})}\, 
 \| V_{j,\perp}^* X( V_j^*  X)^\dagger\|_{2,F}\,
\frac{1}{(1+\gamma_j)^2}
\end{equation}

\begin{equation}\label{suert2}
\| V_{k,\perp}^* \Psi_q (V_k^* \Psi_q)^\dagger\|_{2,F}\leq 
\frac{1}{\psi(\sigma_{k})}\, 
 \| V_{k,\perp}^* X( V_k^*  X)^\dagger\|_{2,F}\, \frac{1}{(1+\gamma_k)^2}
\end{equation}
Moreover, in this case we have 
\begin{equation}\label{suert3}
\frac{\psi(\sigma_{k})}{\psi(\sigma_{j})}\leq 
\frac{1}{(1+\gamma_j)^{2q}} \py
\frac{1}{\psi(\sigma_{k})}\leq 4\, \frac{2^{-2q\,\min\{\sqrt{\gamma_k}\, , \, 1 \}}}{(1+\gamma_k)}\,.
\end{equation}
\end{teo}
\begin{proof} See Section \ref{sec 4.3}.
\end{proof}

\begin{rem}\label{rem still true} Consider the notation in Theorem \ref{third main result S1} above. We point out that the first inequality in Eq. \eqref{suert3} corresponds to a worst
case scenario; moreover, numerical examples show that the upper bound can be (typically) tightened. Indeed, 
let $T_{2q}(x)$ be the Chebyshev polynomial of the first kind of 
degree $2q$ and set $\psi(x)=T_{2q}(x/\sigma_{k+1})$. For $x\geq 1$ we get the representation
$$T_{2q}(x)=\frac{(x+\sqrt{x^2-1})^{2q} + (x-\sqrt{x^2-1})^{2q} }{2}\,.$$
For $1\leq \ell \leq k+1$, set $\eta_\ell=\frac{\sigma_{\ell}}{\sigma_{\ell+1}}\geq 1$ so $\eta_\ell=1+\gamma_\ell$. Then 
$\frac{\sigma_j}{\sigma_{k+1}}=\eta_j\cdot\eta_k>1$ and
$$
\frac{\psi(\sigma_{k})}{\psi(\sigma_{j})}=\frac{T_{2q}(\eta_k)}{T_{2q}(\eta_j\cdot\eta_k)}=
\frac{(\eta_k+\sqrt{\eta_k^2-1})^{2q} + 
(\eta_k-\sqrt{\eta_k^2-1})^{2q}}{\eta_j^{2q}(\eta_k+\sqrt{\eta_k^2-\frac{1}{\eta_j^2}})^{2q} + (\eta_k-\sqrt{\eta_k^2-
\frac{1}{\eta_j^2}})^{2q}}\leq \frac{1}{\eta_j^{2q}}=\frac{1}{(1+\gamma_j)^{2q}}\,.
$$
The upper bound above is sharp since 
$$
\lim_{\eta_k\rightarrow \infty }
\frac{T_{2q}(\eta_k)}{T_{2q}(\eta_j\cdot\eta_k)}
=\frac{1}{(1+\gamma_j)^{2q}}\,.
$$
Nevertheless, in the regime $\eta_j,\,\eta_k\in (1,1+\varepsilon)$ for small $\varepsilon>0$ 
(which is quiet relevant for applications)
we can expect 
better upper bounds. Indeed, numerical examples show that
if $1<\eta_k\leq \eta_j$ then we have that for $q\geq 1$,
\begin{equation}\label{eq conjec1}
\frac{\psi(\sigma_{k})}{\psi(\sigma_{j})}=\frac{T_{2q}(\eta_k)}{T_{2q}(\eta_j\cdot\eta_k)}\leq 
4\,\frac{2^{-2q\,\min\{\sqrt{\gamma_j}\, , \, 1 \}}}{(1+\gamma_j)}\,,
\end{equation}which is a tighter upper bound than that in Eq. \eqref{suert3} for 
$\eta_j\in (1,1+\varepsilon)$ for small $\varepsilon>0$ . We conjecture that Eq. \eqref{eq conjec1} is true (under the previous restrictions). In fact, notice that when $\eta_k\approx 1$ then $\frac{T_{2q}(\eta_k)}{T_{2q}(\eta_j\cdot\eta_k)}\approx \frac{1}{T_{2q}(\eta_j)}$, so Eq. \eqref{eq conjec1} is true (see \cite{D19}). On the other hand,
for $\gamma_j\geq 1$ then Eq. \eqref{suert3} shows that 
$$
\frac{\psi(\sigma_{k})}{\psi(\sigma_{j})}=\frac{T_{2q}(\eta_k)}{T_{2q}(\eta_j\cdot\eta_k)}\leq \frac{1}{(1+\gamma_j)^{2q}}\leq 2^{-2q}
\,.
$$
\end{rem}

In the following result we apply the strategy described in Remark \ref{rem strat} and obtain upper bounds for the convergence of block Krylov subspaces to $h$-dimensional left dominant subspaces
for the spectral norm (the Frobenius norm case can be handled similarly); to simplify the statement below,  we consider the following constants:
given a $h$-compatible pair $(A,X)$ set $C(A,X,j,k)_{2,F}=C_{2,F}$ determined as follows. If $h\leq k<\text{rank}(A)$ and $1\leq j$ then
\beq \label{defi constC}
C_{2,F}(V,X)=C_{2,F}=\max\left \{ 2\,\| V_{j,\perp}^* X( V_j^*  X)^\dagger\|_{2,F}\coma  \| V_{k,\perp}^* X( V_k^*  X)^\dagger\|_{2,F}\right\}\,.
\eeq If $j=0$ we let $C_{2,F}= \| V_{k,\perp}^* X( V_k^*  X)^\dagger\|_{2,F}$; if $k=\text{rank}(A)$ then we set $C_{2,F}=2\,\| V_{j,\perp}^* X( V_j^*  X)^\dagger\|_{2,F}$.

\begin{teo}\label{coro para comp abc1}
Let $(A,X)$ be $h$-compatible, assume that $0\leq j < h \leq k<\text{rank}(A)$, and let $\cK_{q+1}=\cK_{q+1}(A,X)\subset \K^m$, for $q\geq 0$. 
Let $T_{2q}(x)$ be the Chebyshev polynomial of the first kind of 
degree $2q$ and set $\psi(x)=T_{2q}(x/\sigma_{k+1})$. If we let
$\Psi_q=V(\psi(\Sigma)\cdot \Sigma^2)V^*X$ then there exists an $h$-dimensional left dominant subspace $\cS'$ of $A$  such that:
\begin{equation}\label{ecuac pepi2}
\|\sin\Theta (R(A\, \Psi_q),\cS' )\|_{2,F}\leq
C_{2,F} \, \left( \frac{\psi(\sigma_k)}{\psi(\sigma_j)}\,\frac{1}{(1+\gamma_j)^2}
 +  \frac{1}{\psi(\sigma_{k})}\, 
 \frac{1}{(1+\gamma_k)^3}\right)\,,
\end{equation} where $C_{2,F}=C_{2,F}(V,X)$.
In particular, since $R(A\, \Psi_q)\subset \cK_{q+1}$ and $\dim R(A\, \Psi_q)\geq h$,
\beq\label{eq controlsin1}
\|\sin \Theta(\cK_{q+1}, \cS')\|_{2,F}\leq C_{2,F} \,\left(\frac{1}{(1+\gamma_j)^{2(q+1)}} + 
4\,\frac{ 2^{-2q\,\min\{\sqrt{\gamma_k}\, , \, 1 \}}}{(1+\gamma_k)^4} \right)
\,.\eeq
In case $j=0$ then the first terms (in the expressions between parentheses) in Eqs. \eqref{ecuac pepi2} and \eqref{eq controlsin1} should be omitted.
\end{teo}

\begin{proof}
See Section \ref{sec 4.3}.
\end{proof}

As already mentioned, the upper bound in Eq. \eqref{ecuac pepi2} can be obtained by following the three-step strategy in Remark \ref{first main result}. 
The reader can check that in case $k=\text{rank}(A)$ (so that $\sigma_{k+1}=0$) then we can set $\psi(x)=T_{2q}(x/\sigma_k)=T_{2q}(x/\sigma_{j+1})$ and get the bound in Eq. \eqref{ecuac pepi2} omitting the second terms 
in the expressions between parentheses; in this case Eq. \eqref{eq controlsin1}
becomes 
$$\|\sin \Theta(\cK_{q+1}, \cS')\|_{2,F}\leq  \frac{4\, C_{2,F} }{(1+\gamma_j)^3}\, 2^{-2q\,\min\{\sqrt{\gamma_j}\, , \, 1 \}}\,.$$
We point out that Theorem \ref{coro para comp abc1} complements the convergence analysis described in 
\cite[Corollary 2.5]{D19}. That is, this result can be applied in contexts in which 
\cite[Corollary 2.5]{D19} can not be applied; moreover, Theorem \ref{coro para comp abc1} shows the existence of 
$h$-dimensional subspaces $\cT\subseteq \cK_{q+1}=\cK_{q+1}(A,X)$ that are
arbitrarily good approximations of {\it some} $h$-dimensional left dominant subspace $\cS'$ of $A$ (for sufficiently large $q\geq 0$). As opposed to Theorem \ref{first main result}, Theorem \ref{coro para comp abc1} does not require a priori knowledge of $\Theta(R( V_k^* X),V_k^*\cV_j)$ 
(see the comments after Theorem \ref{first main result}). 
On the other hand, the speed at which the upper bound in Eq. \eqref{eq controlsin1} decreases depends on both gaps $\gamma_j$ and $\gamma_k$; in particular, this result warrants a better convergence speed when both singular gaps are significant.

\subsection{Low-rank approximations from block Krylov methods}\label{sec low rank approx}

Notice that 
the upper bounds in Theorem \ref{coro para comp abc1} can be made arbitrarily small for large enough $q\geq 0$. Therefore the corresponding block Krylov subspace contains (arbitrarily good) approximate left dominant subspaces. Still, the previous results do not provide a practical method to compute such approximate dominant subspaces and the corresponding low-rank approximations. In this section we revisit Algorithm \ref{algoalgo} without assuming a singular gap, as a practical way to construct low-rank approximations.  Our approach to deal with this problem is based on the approximate left dominant subspaces of a matrix $ A$; indeed, we follow arguments from \cite{WZZ15}.

For the next results, we consider Algorithm \ref{algoalgo} with input: $ A\in\K^{m\times n}$, starting guess $ X\in \K^{n\times r}$ and power $\ell=q+1$ for some $q\geq 0$;  we set our target rank to $1\leq h\leq $ rank $( A)$.
Once the algorithm stops, we consider the output matrices $U_K$ and $\hat U_h$; we further describe $\hat U_h$ in terms of its columns,
$\hat{ U}_h=(\hat{ u}_1,\ldots,\hat{ u}_h)\in\K^{m\times h}$.  In this case we set
\begin{equation} \label{defi hatUi}
\hat{ U}_i=(\hat{ u}_1,\ldots,\hat{ u}_i)\in\K^{m\times i}\ , \peso{for} 1\leq i\leq h\,.
\end{equation}
As before, we let $j=j(h)<h\leq k=k(h)$ be defined as in Eqs. \eqref{eq defi j} and \eqref{eq defi k},
 and we consider the notation used so far. 
 Further, given $1\leq i\leq h$ we let $ A_i=U_i\,\Sigma_i\,V_i^*\in\K^{m\times n}$ denote a best rank-$i$ approximation of $ A$ (so that $\| A- A_i\|_2=\sigma_{i+1}$).

Let $(A,X)$ be $h$-compatible and let $q\geq 0$. 
In what follows we make use of the constants $\delta(A,X,q,j,k)_{2,F}=\delta_{2,F}$ given by: 
for $k<\text{rank}(A)$ then let $T_{2q}(x)$ be the Chebyshev polynomial of the first kind of degree $2q$, $\psi(x)=T_{2q}(x/\sigma_{k+1})$ and set
\beq\label{eq defi deltai} 
\delta_{2,F}:=\sqrt 2\, 
C_{2,F} \, \left( \frac{\psi(\sigma_k)}{\psi(\sigma_j)}\,\frac{1}{(1+\gamma_j)}
 +  \frac{1}{\psi(\sigma_{k})}\, 
 \frac{1}{(1+\gamma_k)^2}\right)\,,
\end{equation} where $C_{2,F}$ is defined as in Eq. \eqref{defi constC}.
In case $j=0$ the first term should be omitted.
In case $k=\text{rank}(A)$ then let $\psi(x)=T_{2q}(x/\sigma_{k})=T_{2q}(x/\sigma_{j+1})$ and set 
\beq\label{eq defi deltaiF}
\delta_{2,F}:= 
 \,  \frac{\sqrt 2}{\psi(\sigma_j)}\,\frac{C_{2,F}}{(1+\gamma_j)}
\,.\end{equation}

\begin{teo}\label{teo nuevo bb1}
Let $(A,X)$ be $h$-compatible and let $q\geq 0$ be such that $\delta_2\leq 1$. Let $U_K$ be the output of Algorithm \ref{algoalgo} with the power parameter set to $q+1\geq 1$. 
Then there exists an $h$-dimensional right dominant subspace $\tilde \cS$ of $A$ such that 
\beq \label{eq teonuevo1}
\| AP_{\tilde \cS}-U_KU_K^* AP_{\tilde\cS}\|_{2,F}\leq \sigma_{h+1}\cdot\delta_{2,F}\,.
\eeq
\end{teo}
\begin{proof} See Section \ref{sec4.4}.
\end{proof}

Although rather technical, the previous result allows us to get the following estimates for the first $h$ singular values of $A$.

\begin{teo}\label{teo nuevo bb2}
Let $(A,X)$ be $h$-compatible and let $q\geq 0$ be such that $\delta_2\leq 1$. Let $\hat A_h=\hat U_h \hat U_h^* A$, where $\hat U_h$ is the output of Algorithm \ref{algoalgo} with the power parameter set to $q+1\geq 1$. Then, 
\beq\label{eq est for singval}
 \sigma_{i}-\sigma_{h+1}\cdot\delta_2\leq \sigma_i(\hat A_h) \leq \sigma_i\peso{for}  1\leq i\leq h\,.
\eeq
\end{teo}
\begin{proof} See Section \ref{sec4.4}.
\end{proof}

\begin{teo}\label{other main result}
Let $(A,X)$ be $h$-compatible and let $q\geq 0$ be such that $\delta_2\leq 1$.
Let $\hat U_{h}$ be the output of Algorithm \ref{algoalgo} with the power parameter set to $q+1\geq 1$. Then, for every $1\leq i\leq h$ we have that 
\begin{equation}\label{eq conc teo aprox}
\| A-\hat{ U}_i\hat{ U}_i^* A\|_{2,F}\leq \| A- A_i\|_{2,F} + \sigma_{i+1} \cdot \delta_F\,,
\end{equation} 
where $\delta_{2,F}$ are defined in Eq. \eqref{eq defi deltai} or \eqref{eq defi deltaiF}.
\end{teo}
\begin{proof} See Section \ref{sec4.4}.
\end{proof}

Using the previous results together with estimates for $\delta_{2,F}$ as defined in Eqs. \eqref{eq defi deltai} and \eqref{eq defi deltaiF} we can obtain the following upper bounds for the approximation of singular values and approximation errors using the block Krylov Algorithm \ref{algoalgo}.

\begin{teo}\label{teo supermain comple1}
Let $(A,X)$ be $h$-compatible, let $0\leq j<h\leq k<\text{rank}(A)$, and let $C_{2,F}$ be  defined as in Eq. \eqref{defi constC}. Assume that 
\begin{equation}\label{eq cond conc1}
\sqrt{2}\,C_{2}\,\left( \frac{1}{(1+\gamma_j)^{2q+1}}+ 4\,\frac{2^{-2q\,\min\{\sqrt{\gamma_k}\, , \, 1 \}}}{(1+\gamma_k)^3}\right)\leq 1\,,
\end{equation}
where we follow the convention: $\gamma_0=+\infty$ in case $j=0$. 
Let $\hat U_h$ be the output of Algorithm \eqref{algoalgo} with the power parameter set to $q+1\geq 1$;
Let 
$$\hat A_i=\hat U_i\hat U_i^*A\peso{for} 1\leq i\leq h\,,$$ 
be the rank-$i$ approximation of $A$ obtained from Algorithm \ref{algoalgo}. Then, for $1\leq i\leq h$:
$$\sigma_i-\sigma_{h+1}\cdot \sqrt{2}\,C_{2}\,\left( \frac{1}{(1+\gamma_j)^{2q+1}}+ 4\,\frac{2^{-2q\,\min\{\sqrt{\gamma_k}\, , \, 1 \}}}{(1+\gamma_k)^3}\right)
\leq \sigma_i(\hat A_h)\leq \sigma_i$$ 
and 
$$\| A-\hat{ U}_i\hat{ U}_i^* A\|_{2,F}\leq \| A- A_i\|_{2,F} + \sigma_{i+1} \cdot \sqrt{2}\,C_{F}\,\left( \frac{1}{(1+\gamma_j)^{2q+1}}+ 4\,\frac{2^{-2q\,\min\{\sqrt{\gamma_k}\, , \, 1 \}}}{(1+\gamma_k)^3}\right)
\,. $$
\end{teo}

\begin{proof}
See Section \ref{sec4.4}.
\end{proof}

Theorem \ref{teo supermain comple1} complements the analysis in the deterministic setting obtained in 
\cite[Corollary 2.6]{D19}.
That is, this result can be applied in contexts in which 
\cite[Corollary 2.6]{D19} can not be applied (i.e., when there is no singular gap). Theorem \ref{teo supermain comple1} shows that the distance $\|A-\hat A_h(q)\|_{2,F}$, between the low-rank approximations $\hat A_h(q)$ of $A$ obtained from the block Krylov method and $A$, converges to the minimial distance $\|A-A_h\|_{2,F}$ as the power parameter (number of iterations) $q\geq 0$ increases. Moreover, it also shows the convergence of the first $h$ singular values of  $\hat A_h$ to the corresponding singular values of $A$.
Thus, our approach follows the paradigm introduced in \cite{MM15} for the assessment of the quality of low-rank approximations. On the other hand, the speed at which the upper bounds  in Theorem \ref{teo supermain comple1} decrease depends on both gaps $\gamma_j$ and $\gamma_k$; that is, this result warrants better convergence speeds when both singular gaps are significant. 

\begin{rem}
There is a vast literature related to low-rank approximation, both from a deterministic and randomized point of view, taking into account singular gaps, or disregarding these gaps (see the survey \cite{KiSc17}). 

Randomized methods \cite{DKM06b,DM16,Gu15,HMT11,MM15} typically draw a random $n\times r$ matrix $ X$ (a starting guess matrix) and consider the random subspace $R( X)\subset \K^n$ given by the range of $ X$. One of the advantages of this approach is that it is possible to prove that, with high probability, $X$ satisfies some compatibility properties with the structure of $A$, regardless of the particular choice of $A$ (see \cite{DI19}); on the other hand, the conclusions of the randomized approach then typically hold  with high probability. Thus, these conclusions do not apply to particular choices of matrices $X$.  In this setting, the work \cite{MM15} contains a deep analysis of the block Krylov method for low-rank approximations, when the starting guess matrix is drawn from a standard Gaussian random matrix. The gap-independent bounds obtained in \cite[Theorems 10 and 12]{MM15} (see also \cite[Theorems 11 and 12]{MM15}) imply that a number
$O(1/\sqrt{\varepsilon})$ of iterations in Algorithm \ref{algoalgo} are required to warrant (with high probability) that $\|A-\hat A_h(q)\|_2\leq (1+\varepsilon) \|A-A_h\|_2$ and $|\sigma_i(A)^2 - \sigma_i(\hat A_h)^2|\leq \varepsilon\,\sigma_h(A)^2 $, for $1\leq i\leq h$.

Notice that we follow a deterministic approach where we consider a fixed (deterministic) $n\times r$ matrix $X$ satisfying a set of specific (explicit) compatibility hypotheses with the structure of $A$ (that hold in generic cases) and obtain conclusions that hold for the particular matrices $X$ and $A$.  In our deterministic setting, 
Theorem \ref{teo supermain comple1}  implies (for a fixed $h$-compatible pair $(A,X)$ and $1\leq j<h\leq k<\text{rank}(A)$) that 
$$
\| A-\hat{ U}_i\hat{ U}_i^* A\|_{2,F}\leq \| A- A_i\|_{2,F} + \sigma_{i+1} \cdot 2\,\sqrt{2}\,C_{F}\,\max \left\{ \frac{1}{(1+\gamma_j)^{2q+1}}\coma 4\,\frac{2^{-2q\,\min\{\sqrt{\gamma_k}\, , \, 1 \}}}{(1+\gamma_k)^3}\right\}\,,
$$where $C_{F}=\max\{ 2\,\| V_{j,\perp}^* X( V_j^*  X)^\dagger\|_{F}\coma  \| V_{k,\perp}^* X( V_k^*  X)^\dagger\|_{F}\}$. Hence, if $q$ is the smallest integer satisfying
$$
\log(\frac{2\,\sqrt{2}\,C_F}{\varepsilon})\leq \min\left \{ (2q+1)\,\log(1+\gamma_j)\coma 2\,\log(2)\,q\,\min\{\sqrt{\gamma_k},1\} \right\}
$$
then
Algorithm \ref{algoalgo} provides a low-rank approximation of $A$ such that  $\|A-\hat A_h(q)\|_2\leq (1+\varepsilon)\,\|A-A_h\|_2$, $|\sigma_i(A)^2 - \sigma_i(\hat A_h)^2|\leq \varepsilon\,\sigma_h(A)^2 $, for $1\leq i\leq h$. Thus Theorem \ref{teo supermain comple1} implies that, (for fixed $A$ and $X$), a number $
O(\, (\min\{ \,2\,\log(1+\gamma_j)\coma 2\,\sqrt{2}\,\min\{\sqrt{\gamma_k},1\}\,\} ) ^{-1}\,\log(\frac{2\,\sqrt{2}\,C_F}{\varepsilon})\, )$
of iterations in Algorithm \ref{algoalgo} are required to achieve the  previously mentioned approximation errors.
On the other hand, we remark that it is well known that an analysis of convergence such as that derived in Theorem \ref{coro para comp abc1} (for dominant subspaces) or Theorem \ref{teo supermain comple1} can be used to obtain upper bounds for the approximation errors in the randomized setting (for example, see Saibaba's work in the context of the Subspace Iteration Method \cite{Sai19}). We will consider these issues elsewhere.
\end{rem}

\begin{rem}\label{rem coments finales}
As a final comment, we point out that the present results together with the results from \cite{D19}
do not cover the {\it complete picture} of the convergence analysis of deterministic block Krylov methods. 
For example, assume that we are interested in computing an approximation of an $h$-dimensional dominant 
subspace of the matrix $A$ in the case that there is a {\it small} singular gap $\sigma_{h}>\sigma_{h+1}$, for which 
$\gamma_h\approx 0$ (so $\frac{\sigma_h}{\sigma_{h+1}}\approx 1$). This situation corresponds, for example, to the case where $\sigma_h$ lies in a cluster of singular values $\sigma_{j+1}>\ldots>\sigma_h>\ldots>\sigma_k$, for $1\leq j<h<k<\text{rank}(A)$ and $\sigma_{j+1}-\sigma_{k}\approx 0$. In this case 
\cite[Theorem 2.1]{D19} exhibits a rather slow speed of convergence with respect to the power parameter (number of iterations) $q\geq 0$. 
Although our approach does not apply directly to the case where $\sigma_h$ lies in a cluster of singular values, it seems that our results and techniques could shed some light on the convergence analysis in this case. Indeed, let us consider the following {\it heuristic argument}: let $A=U\Sigma V^*$ and
assume that  $\sigma_{j+1}\geq \ldots\geq \sigma_h\geq \ldots\geq \sigma_k>0$, for $1\leq j<h<k$ and set 
$\varepsilon=\sigma_{j+1}-\sigma_{k}\approx 0$. Let us construct an auxiliary matrix 
$$A_{\rm aux}=U\Sigma_{\rm aux} V^* \ , \ \ \Sigma_{\rm aux}=\text{diag}(\sigma_1,\ldots,\sigma_j, c,\ldots, c,\sigma_{k+1},\ldots,\sigma_p)$$ where $p=\min\{m,\,n\}$ and $c=1/2(\sigma_{j+1}+\sigma_{k})$. 
We will use the key fact that the dominant subspaces $\cU_j$ and $\cU_k$ of $A$ and $A_{\rm aux}$ of dimensions $j$ and $k$, respectively, coincide (although $h$-dimensional dominant subspaces of $A$ and $A_{\rm aux}$ do not necessarily coincide). Similarly, the gaps $\gamma_j,\,\tilde \gamma_j>0$ and $\gamma_k,\,\tilde \gamma_k>0$ of $A$ and $A_{\rm aux}$ are {\it comparable}. 

Let $T_{2q}(x)$ be the Chebyshev polynomial of the first kind of 
degree $2q$ and set $\psi(x)=T_{2q}(x/\sigma_{k+1})$ and let
$\Psi_q=V(\psi(\Sigma)\cdot \Sigma^2)V^*X$. Then, by
Theorem \ref{third main result S1} we get that 
$\|\sin\Theta(R(\Psi_q)),\cV_j)\|_{2,F}$ and $\| V_{k,\perp}^* \Psi_q (V_k^* \Psi_q)^\dagger\|_{2,F}$ 
decay (as a function of the number of iterations) in terms of the enveloping gaps $\gamma_j$ and $\gamma_k$. 
Hence, if we now apply Theorem \ref{first main result} to the auxiliary matrix $A_{\rm aux}$ with starting guess $\Psi_q$ and with $t=0$, we see that there exists an $h$-dimensional left dominant subspace $\cS_{\rm aux}'$ of $A_{\rm aux}$ such that 
$$
\|\sin \Theta(R(A_{\rm aux} \Psi_q),\cS_{\rm aux}')\|_{2,F}\leq 2\|\sin \Theta(R(\Psi_q)),\cV_j)\|_{2,F}+\frac{1}{1+\tilde \gamma_k} \|V_{k,\perp}^*\Psi_q\,(V_k^* \Psi_q)^\dagger\|_{2,F}\,.
$$
It seems reasonable to expect that the difference between $\Theta(R(A \Psi_q),\cS_{\rm aux}')$ and $\Theta(R(A_{\rm aux} \Psi_q),\cS_{\rm aux}')$ can be bounded in terms of some 
measure of the {\it spread} of the cluster $ \sigma_j\geq \ldots\geq \sigma_k$. Assuming this last fact and 
putting together the previous estimates, we now conclude that 
$$\|\sin \Theta(\cK_{q+1}, \cS_{\rm aux}')\|_{2,F} \leq \|\sin \Theta(R(A \Psi_q),\cS_{\rm aux}')\|_{2,F}$$
where the upper bound would decay in terms of the gaps $\gamma_j$ and $\gamma_k$ and the power parameter; here we have used that $R(A\Psi_q)\subset \cK_{q+1}=R(K_{q+1}(A,X))$. 

At this point, we remark that the subspaces $\cS_{\rm aux}'=\cS_{\rm aux}'(q)$ are
 not necessarily  dominant subspaces of $A$. Yet, the subspaces $\cS_{\rm aux}'$ satisfy that there exist
 $(h-j)$-dimensional subspaces $\cU\subset \cU_k\ominus \cU_j$ such that $\cS_{\rm aux}'=\cU_j\oplus \cU$. Using this last fact we could show that $\|A-P_{\cS_{\rm aux}'} A\|_{2,F}$ can be bounded by $\|A-A_j\|_{2,F}+\|(\sigma_{\ell+j}-c)_{\ell=1}^{k-j}\|_{2,F}$, where $\|A-A_j\|_{2,F}\approx \|A-A_h\|_{2,F}$ and $\|(\sigma_{\ell+j}-c)_{\ell=1}^{k-j}\|_{2,F}\approx 0$, because we are assuming that $\sigma_{j+1}\geq \ldots\geq \sigma_k$ forms a cluster of singular values. That is, even when $\cS_{\rm aux}'$ 
is not an $h$-dimensional left dominant subspace of $A$, it would behave much like one. 
 Thus, in case $\gamma_h>0$ and $\gamma_h\approx 0$ but with significant gaps $\gamma_j$ and $\gamma_k$, we would conclude the existence of the subspaces $\cS_{\rm aux}'(q)$ that become close to $\cK_{q+1}$ rather fast, as $q\geq 0$ increases. In this case $\cS_{\rm aux}'$ could be used to study the low-rank approximations obtained from Algorithm \ref{algoalgo} (similarly to the way that we have used the existence of a dominant subspace of $A$ lying close to $\cK_{q+1}$ to study low-rank approximations obtained from Algorithm \ref{algoalgo}). We will consider this type of analysis elsewhere.
\end{rem}

\section{Proofs of the main results}\label{sec proofs}

In this section we present detailed proofs of our main results. Some of our arguments make use of some
basic facts from matrix analysis, that we develop in Section \ref{apendixity} (Appendix). We begin by recalling the notation introduced so far; then we consider some general facts about block Krylov spaces, principal angles and principal vectors between subspaces that are needed for developing the proofs below.

\medskip

\begin{nota}\label{nota21} \rm We keep the notation and assumptions introduced so far; hence, we consider: 
\ben
\item $A\in\K^{m\times n}$ with singular values $\sigma_1\geq \ldots\sigma_p\geq 0$, with $p=\min\{m,n\}$.
\item $1\leq h\leq \text{rank}( A)\leq p$; moreover, 
we let $0\leq j(h)< h$ be given by
$$j=j(h)=\max\{ 0\leq \ell < h\ : \ \sigma_\ell> \sigma_h\}<h$$ where we set $\sigma_0=+\infty$ and 
$$k=k(h)=\max\{ 1\leq \ell\leq \text{rank}( A) \, : \, \sigma_\ell=\sigma_h\}\geq h\,.$$ 
\item A starting guess $X\in\K^{n\times r}$ such that $(A,X)$ is $h$-compatible; that is, we assume that
there exists
an $h$-dimensional right dominant subspace $\cS$ of $ A$ such that $ \Theta(R( X),\cS)<\frac{\pi}{2} I$. In this case, $\dim X^*\cS=h$; in particular $r\geq \text{rank}(X)\geq h$.
\item  For any $\ell\geq 0$ we consider $K_\ell=K_\ell(A,X)$
constructed in terms of $A$ and $ X$ as in Eq. \eqref{eq defi Krylov2}, that is
$$
K_\ell=K_\ell( A, X)=  (\,A  X \quad ( A  A^*) A  X\quad \ldots\quad 
( A  A^*)^\ell A  X\,) \in  \K^{m\times (\ell+1)r}
$$
 and consider the block Krylov space $\cK_\ell=\cK_\ell( A, X)=R(K_\ell)$.
\item $A=U  \Sigma  V^*$ a SVD of $A$. Given $1\leq \ell\leq \text{rank}(A)$ we consider the partitions
\begin{equation}\label{eq decomp sigma y otros}
 \Sigma=\begin{pmatrix}  \Sigma_\ell & \\ &  \Sigma_{\ell,\perp}\end{pmatrix}\ , \ \ 
 U=\begin{pmatrix}  U_\ell & 
 U_{\ell,\perp}\end{pmatrix}\ , \ \ 
 V=\begin{pmatrix}  V_\ell & 
 V_{\ell,\perp}\end{pmatrix}\,.
\eeq
Since the pair $(A,X)$ is $h$-compatible then we further assume that $R(V_h^*X)=R(V_h)$; this can always be done by choosing a convenient SVD of $A$.
\een
\end{nota}

 Notice that for every polynomial $\varphi(x)\in \K[x]$ of degree at most $\ell$ we get that the range of 
$\varphi( A  A^*) A  X\in\K^{m\times r}$ is contained in $\cK_\ell$. In terms of a SVD of $ A$, we get that 
$$
\varphi( A  A^*) A  X=  U \varphi( \Sigma^2 ) \Sigma  V^* X= U \phi( \Sigma ) V^* X
$$ where $\phi(x)=x\,\varphi(x^2)\in\K[x]$ is a polynomial of degree at most $2\ell+1$ with odd powers only, and  represents a generalized matrix function (see \cite{ABF16,HBI73}). Here $ \Sigma =\text{diag}(\sigma_1,\ldots,\sigma_p)\in\R^{m\times n}$, where $p=\min\{m,n\}$; hence, $$\phi( \Sigma)=\text{diag}(\phi(\sigma_1),\ldots,\phi(\sigma_p))\in\K^{m\times n}\,.$$ In this case we write
\begin{equation}\label{eq defi super fi1}
 \Phi:= U\phi( \Sigma) V^* X\in \K^{m\times r}\,,
\end{equation} so by the previous facts, $R( \Phi)\subset\cK_\ell$. 
Let $\cS$ be an $h$-dimensional right dominant subspace $\cS$ of $ A$ such that $ \Theta(R( X),\cS)<\frac{\pi}{2} I$ as above. As already mentioned, we can consider a SVD of $ A=U\Sigma V^*$ in such a way that $\cS=\cV_h$. In this case, $R( V_h^* X)=R( V_h^*)$: if we assume further that $\phi(\sigma_i)\neq 0$ for $1\leq i\leq h$ (so $\phi(\Sigma_h)\in \K^{h\times h}$ is an invertible matrix) then $\dim R( \Phi)\geq h$, where $ \Phi$ is defined as in Eq. \eqref{eq defi super fi1}.
We will further consider similar facts related to convenient block decompositions of a SVD of $ A$.

\begin{rem}\label{rem sobre ang princ}
We mention some further properties of the principal angles between subspaces (see Section \ref{sec aux angles} for definitions) that we will need in what follows.  Let $\cS,\,\cT\subset \K^n$ be two subspaces such that $s=\dim\cS\leq \dim\cT=t$.
If $\cS'\subset\cS$ and $\cT\subset \cT'$ are subspaces 
with dimensions $s'$ and $t'$ respectively, then 
$$
\|  \Theta(\cS,\cT')\|_{2,F} \leq \| \Theta(\cS,\cT)\|_{2,F} \ \ , \ \ 
\| \sin  \Theta(\cS,\cT')\|_{2,F} \leq \|\sin  \Theta(\cS,\cT)\|_{2,F} 
$$ and similarly 
$$\| \Theta(\cS',\cT)\|_{2,F}\leq \| \Theta(\cS,\cT)\|_{2,F} \ \ , \ \ 
\|\sin  \Theta(\cS',\cT)\|_{2,F}\leq \|\sin  \Theta(\cS,\cT)\|_{2,F} \,,
$$
which follow from the definition of principal angles (see Eq. \eqref{eq sobre sen ang princ1}). On the other hand, $\dim\cS^\perp=n-s\geq n-t=\dim\cT^\perp$ and therefore, for $1\leq j\leq \min\{s,n-t\}$,
$$
\sin(\theta_{(n-t)-j+1}(\cS^\perp,\cT^\perp))=\sigma_j(( I- P_{\cS^\perp}) P_{\cT^\perp})= \sigma_j( P_\cS( I-  P_{\cT}))= \sin(\theta_{s-j+1}(\cS,\cT))\,. 
$$
By the previous identity, we see that the positive angles between 
$\cS$ and $\cT$ coincide with the 
positive angles between 
$\cS^\perp$ and $\cT^\perp$
 Notice that as a consequence of this last fact 
 we get that 
\beq\label{ang inv por com ort2}
\|  \Theta(\cS,\cT)\|_{2,F}=\| \Theta(\cS^\perp,\cT^\perp)\|_{2,F}\,.
\eeq
\end{rem}

\begin{rem}[Principal vectors between subspaces]\label{rem constr princ vect}  In what follows we shall also make use of the principal vectors associated with the subspaces $\cS,\, \cT\subset \K^n$ such that $s=\dim\cS\leq \dim\cT=t$: indeed, by the construction of the principal angles, we get that there exist orthonormal systems
$\{ u_1,\ldots, u_s\}\subset \cS$ and $\{ v_1,\ldots, v_s\}\subset \cT$ 
such that 
$$\langle  u_i, v_j\rangle =\delta_{ij}\,\cos(\theta_{j}(\cS,\cT))\peso{for} 1\leq i,\,j\leq s\,,$$ where $\delta_{ij}$ is Kronecker's delta function.
We say that $\{ u_1,\ldots, u_s\}$ and $\{ v_1,\ldots, v_s\}$ are the {\it principal vectors (directions)} associated with the subspaces $\cS$ and $\cT$. Notice that the previous facts imply, in particular, that the subspaces $\cS_j=\text{Span}\{ u_1,\ldots, u_j\}\subset\cS$ and 
$\cT_j=\text{Span}\{ v_1,\ldots, v_j\}\subset \cT$ are such that 
$$ \Theta(\cS_j,\cT_j)=\text{diag}(\theta_1(\cS,\cT),\ldots,\theta_j(\cS,\cT))\in\R^{j\times j}\peso{for} 1\leq j\leq s\,.$$ 
In this case we say that $\{ u_1,\ldots,u_j\}\subset \cS$ are {\it principal vectors corresponding to } $(\cS,\cT,j)$ (here it is implicit that $1\leq j\leq s=\min\{\dim \cS,\,\dim \cT\}$).
If $\tilde {\cS}\subset \cS$ and $\tilde{\cT}\subset\cT$ are two $j$-dimensional subspaces then, it follows that 
$ \Theta(\cS_j,\cT_j)\leq  \Theta(\tilde {\cS},\tilde {\cT})$; that is, $\cS_j$ and $\cT_j$ are $j$-dimensional subspaces of $\cS$ and $\cT$ respectively, that are at minimal angular (vector-valued) distance.
\end{rem}

\subsection{Proof of Theorem \ref{first main result} 
}\label{sec prueba teo 2.1}

We now present a proof of Theorem \ref{first main result}.
We divide our arguments into steps.

\begin{proof} {\it Step 1: adapting the DIKM-I theory to the present context}. 
Consider Notation \ref{nota21}. 
By construction $\sigma_{j}> \sigma_{j+1}=\sigma_h= \sigma_{k}$. We first assume that $1\leq j$ and $k<\text{rank}( A)\leq p=\min\{m,n\}$.  Since
$k<\text{rank}( A)$ then $\sigma_h= \sigma_{k}>\sigma_{k+1}>0$. 

We consider $ X\in \K^{n\times r}$ such $\dim(\cX)=s\geq h$ and such that 
there exists a right dominant subspace $\cS\subset\K^n$ of dimension $h$ with
$ \Theta(\cS,\cX)< \pi/2\, I$, where $\cX=R( X)\subset \K^n$ denotes the range of $ X$. 
Consider 
 $ A=  U  \Sigma  V^*$ a full SVD. 
We now consider the partitioning as in Eq. \eqref{eq decomp sigma y otros} corresponding to the index $1\leq k\leq \text{rank}(A)$.
It is worth noticing that $ \Sigma_k$, $R( U_k)=\cU_k$ and $R( V_k)=\cV_k$ do not depend on the particular choice of a SVD of $ A$; also notice that the partition is well defined since $k<\text{rank}( A)$. 
 Let $\phi(x)$ be a polynomial of degree $2t+1$ with odd powers only, such that $\phi(\sigma_1),\, \ldots,\, \phi(\sigma_k)> 0$; hence $\phi( \Sigma_k)$ is invertible.

\smallskip

\noindent {\it Step 2: applying the DIKM-I theory to the adapted model}. We let $\cK_t=\cK_t( A, X)$ denote the block Krylov subspace and let $ P_t\in\K^{m\times m}$ denote the orthogonal projection onto $\cK_t$. Notice that if we let $ \Phi\in\K^{m\times r}$ be as in Eq. \eqref{eq defi super fi1}
then $R( \Phi)\subset \cK_t$. 
Consider for now an arbitrary $h$-dimensional subspace $\cS'\subset\K^m$.
Then 
\begin{equation} \label{eq teo 1 31}
\|\sin  \Theta	(\cK_t,\cS')\|_{2,F}=\|(I- P_t) P_{\cS'}\|_{2,F}\leq 
\|(I- \Phi  \Phi^\dagger ) P_{\cS'}\|_{2,F}
\,,\end{equation}
where we have used that $\dim\cK_t\geq \dim R( \Phi)\geq \dim{\cS'}=h$.
We now consider the decomposition $ \Phi= \Phi_k+ \Phi_{k,\perp}$, where
$$
 \Phi_k\equiv U_k \phi( \Sigma_k) V_k^*  X
\py 
 \Phi_{k,\perp}\equiv U_{k,\perp} \phi( \Sigma_{k,\perp}) V_{k,\perp}^*  X\,.
$$
By \cite[Lemma 4.2]{D19} (see also \cite{Maher92}) we get that 
$$\|(I- \Phi  \Phi^\dagger ) P_{\cS'}\|_{2,F}\leq \| P_{\cS'}- \Phi B\|_{2,F}\peso{for}  B\in \K^{r\times m}\,.$$
By the previous inequality we get that 
\begin{equation}\label{eq comp con fik} 
\|(I- \Phi  \Phi^\dagger ) P_{\cS'}\|_{2,F}\leq 
\|(I- \Phi  \Phi_k^\dagger ) P_{\cS'}\|_{2,F}\,.
\end{equation}
We can further estimate
\begin{equation}\label{eq comp con fikx25} 
\|(I- \Phi  \Phi_k^\dagger ) P_{\cS'}\|_{2,F}\leq 
\|(I- \Phi_k  \Phi_k^\dagger ) P_{\cS'}\|_{2,F}+\| \Phi_{k,\perp}  \Phi_k^\dagger  P_{\cS'}\|_{2,F}\,.
\end{equation}
{\it Step 3: dealing with the fact that $R( V_k^* X)\neq R( V_k^*)$}. 
We now consider the two terms to the right of Eq. \eqref{eq comp con fikx25}.
In our present case, we have to deal with the fact that $R( V_k ^* X)\neq R( V_k^*)$ when $h<k$.
Indeed, since $ \Theta(\cS, \cX)<\pi/2\, I$ and $\cS\subset R( V_k)$ we see that if we let 
$$\cW\equiv R( V_k^* X)= V_k^*\cX\subset \K^k$$ then
$k\geq \dim(\cW):= d\geq h$. 
Let $$\cT=\phi( \Sigma_k )\cW\subset \K^k\,.$$ 
Since, by hypothesis, $\phi( \Sigma_k)\in\R^{k\times k}$ is an invertible matrix then $\dim \cT=d$ and
\begin{equation}\label{eq ran fik}
  \Phi_k  \Phi_k^\dagger = U_k P_\cT U_k^*\,.
\end{equation}
We now consider $\cH'=\text{Span}\{ e_1,\ldots, e_j\}\subset \K^k$, where $j=j(h)\geq 1$ and $\{ e_1,\ldots, e_k\}$ denotes the canonical basis of $\K^k$; notice that $\cH'=R(V_k^*V_j)$. We also consider the principal angles 
$$
 \Theta(\cW,\cH')=\text{diag}(\theta_1(\cW,\cH'),\ldots,\theta_j(\cW,\cH'))\in\R^{j\times j}\,.
$$By Proposition \ref{pro app ang entre subs} we get that 
$$
 \Theta(\cW,\cH')\leq
  \Theta(\cX,\cV_j) <\frac{\pi}{2}\, I\,,
$$
since $\cV_j\subset R( V_k)$ and $ V_k^* \cV_j=\cH'$, and the second inequality above follows from the fact that $\theta_i(\cX,\cV_j)\leq \theta_i(\cX,\cS)<\frac{\pi}{2}$, for $1\leq i\leq j$, since $\cV_j\subset \cS$ (see Section \ref{sec aux angles}).

\smallskip

\noindent {\it Step 4: computing the left dominant subspace $\cS'$}. 
Let $\{ w_1,\ldots, w_j\}\subset\cW$ and $\{ f_1,\ldots, f_j\}\subset \cH'$ be the principal vectors associated with $\cW$ and $\cH'$ (as described in Section \ref{sec aux angles}).
Let $\cW'=\text{Span}\{ w_1,\ldots, w_j\}\subset \cW$; in this case, $ \Theta(\cW,\cH')= \Theta(\cW',\cH')$, by construction. 
Consider the subspace $\cT'=\phi( \Sigma_k)\,\cW'\subset \cT$ so $\dim(\cT')=\dim(\cW')=j=\dim(\cH')$; since $\cH'$ is an invariant subspace of $\phi( \Sigma_k)$ then Proposition \ref{pro app 2} implies that 
$$
\|\sin \Theta(\cT',\cH')\|_{2,F}\leq 
\| \sin \Theta(\cW',\cH')\|_{2,F}=\| \sin \Theta(\cW,\cH')\|_{2,F}
$$
since 
 $\|\phi( \Sigma_k) ( I- P_{\cH'})\|_2\,\|\phi( \Sigma_k)^{-1}\|_2=1$, where we used that $\phi(\sigma_i)\geq \phi(\sigma_k)>0$, for $1\leq i\leq k$ and that $\phi(\sigma_{j+1})=\phi(\sigma_{k})$. 

\smallskip

Let $\cT''=\cT\ominus\cT'$ so $\dim \cT''=d-j$ and $\cT''\subset (\cT')^\perp$.
Since $\dim((\cT')^\perp)= \dim((\cH')^\perp)$, by Eq. \eqref{ang inv por com ort2} we see that 
$$
\| \Theta(\cT'',(\cH')^\perp)\|_{2,F}\leq 
\| \Theta((\cT')^\perp,(\cH')^\perp)\|_{2,F}=\| \Theta(\cT',\cH')\|_{2,F}\leq
\| \Theta(\cW,\cH')\|_{2,F}\,.
$$
Let $\{ y_1,\ldots, y_{d-j}\}\subset \cT''$ and $\{ z_1,\ldots, z_{d-j}\}\subset (\cH')^\perp$ be the principal vectors associated with $\cT''$ and $(\cH')^\perp$.
Then, if we let $\cH''=\text{Span}\{ z_1,\ldots, z_{h-j}\}$ we have that $\dim\cH''=h-j$,
$$\|\sin  \Theta (\cT'',\cH'')\|_{2,F}\leq \|\sin  \Theta (\cT'',(\cH')^\perp)\|_{2,F}\leq \|\sin  \Theta(\cW,\cH')\|_{2,F}\,. $$
On the one hand, we have that $\cT=\cT'\oplus \cT''$; on the other hand, we have that 
$$
\cS':= U_k (\cH'\oplus \cH'' )=\cU_j\oplus   U_k \cH'' \subseteq \cU_k \subseteq \K^m
$$is an $h$-dimensional left dominant subspace of $A$ (see Section \ref{sec dom subs3}). 

\medskip

\noindent{\it Step 5: obtaining some more upper bounds}. 
Since
$$
\|\sin  \Theta (\cT',\cH')\|_{2,F}\coma\| \sin  \Theta (\cT'',\cH'')\|_{2,F}\leq\| \sin  \Theta(\cW,\cH')\|_{2,F}
$$
then Proposition \ref{prop pegoteo de subespacios} implies that 
$\|\sin  \Theta( \cT\coma \cH'\oplus\cH'' )\|_{2,F} \leq 2\,
\| \sin  \Theta(\cW,\cH') \|_{2,F}$.
Hence
\begin{equation}\label{eq teo1 41}
\|(I- \Phi_k \Phi_k^\dagger)  P_{\cS'} \|_{2,F}=\|\sin  \Theta(  U_k\cT\coma \cS' )\|_{2,F} \leq 2\,
\| \sin  \Theta(\cW,\cH') \|_{2,F}\,,
\end{equation}
since $ U_k$ is an isometry and $R( \Phi_k)= U_k\,\cT$ (see Eq. \eqref{eq ran fik}).

\smallskip

By Proposition \ref{pro app 1}, since $\cW=R( V_k^* X)$,
$$
 \Phi_k^\dagger = ( V_k^*  X) ^\dagger ( U_k \phi( \Sigma_k) P_\cW)^\dagger\,.$$
Since $ U_k\in \K^{m\times k}$ has a trivial kernel, we get that 
$$
( U_k \phi( \Sigma_k) P_\cW)^\dagger=(\phi( \Sigma_k) P_\cW)^\dagger ( U_k  P_\cT)^\dagger=(\phi( \Sigma_k) P_\cW)^\dagger  P_\cT U_k^*=(\phi( \Sigma_k) P_\cW)^\dagger  U_k^*$$ 
where we have used Proposition \ref{pro app 1}, that 
$( U_k  P_\cT)^\dagger=( U_k  P_\cT)^*=  P_\cT U_k^*$ since
$  U_k  P_\cT$ is a partial isometry and that $\ker((\phi( \Sigma_k) P_\cW)^\dagger)^\perp=\cT$. The previous facts show that 
$$
 \Phi_{k,\perp}  \Phi_k^\dagger  P_{\cS'} =
 U_{k,\perp} \phi( \Sigma_{k,\perp}) V_{k,\perp}^*  X
( V_k^*  X) ^\dagger (\phi( \Sigma_k) P_\cW)^\dagger  U_k^* 
 P_{\cS'}
$$ so then,
\begin{equation}\label{eq teo1 42}
\| \Phi_{k,\perp}  \Phi_k^\dagger  P_{\cS'} \|_{2,F}\leq \|\phi( \Sigma_{k,\perp})\|_2\,\|\phi( \Sigma_k)^{-1}\|_2\,\|
 V_{k,\perp}^*  X
( V_k^*  X) ^\dagger \|_{2,F}\,.
\end{equation} The result now follows from the estimates in Eqs. \eqref{eq teo 1 31}, \eqref{eq comp con fik} and \eqref{eq comp con fikx25} 
together  with the bounds in Eqs. \eqref{eq teo1 41} and \eqref{eq teo1 42}.

The cases in which $j=0$ or $k=\text{rank}( A)$ can be dealt with similar arguments. Indeed, notice that if $j=0$ then we can take $\tilde \cT\subset \cT$ such that $\dim\tilde \cT=h$, and set $\cS'= U_k \tilde \cT$. By construction, $\cS'\subset R( \Phi_k)$ is a left dominant subspace of $ A$ (in this case any subspace of $\cU_k$ is a dominant subspace of $ A$). Finally, in case $k=\text{rank}( A)$ then $ \Sigma_{k,\perp}=0$ and then $\phi( \Sigma_{k,\perp})=0$,
so that we also get $ \Phi_{k,\perp}=0$.
\end{proof}

Some comments related to the previous proof are in order. We have followed the general lines of the proof of \cite[Theorem 2.1]{D19}. Nevertheless, the assumption in \cite{D19} (i.e., that $R( V_k^* X)= R( V_k^*)$) automatically implies that 
$\|(I- \Phi_k  \Phi_k^\dagger ) P_{\cS'}\|_{2,F}=0$ in Eq. \eqref{eq comp con fik}.
Since we are only assuming that the pair $( A, X)$ is $h$-compatible, our arguments need to include Steps 3, 4 and the first part of Step 5.

With the notation of the proof, a careful inspection of the arguments above shows that the subspace $\cS'$ can be computed explicitly (in terms of several objects that we describe below). We include the following pseudo-code of the construction of $\cS'$, for the convenience of the reader. In the code, we invoke the construction (in terms of convenient singular value decompositions) of principal vectors corresponding to triplets $(\cS,\cT,j)$ as considered in Remark \ref{rem constr princ vect}. We keep using Notation \ref{nota21}.

\smallskip

\begin{algorithm}
\caption{(Construction of $\cS'$ as in Theorem \ref{first main result})}\label{algoalgo2}
\centerline{
}
\begin{algorithmic}[1]
\REQUIRE  $A=U\Sigma V^* \in\K^{m\times n}$ a SVD of $A$, (a compatible) starting guess $ X\in \K^{n\times r}$; rank parameters $0\leq j<h<k\leq \text{rank}(A)$;
power parameter $t\geq 0$; a polynomial $\phi(x)\in\K[x]$ of degree at most $2t+1$ with odd powers only, such that $\phi(\sigma_1),\,\ldots,\,\phi(\sigma_k)>0$.
 Set 
$$
 \cW=R(V_k^*X) \ \ , \ \ \cT=\phi(\Sigma_k)\cW \ \ , \ \ \cH'=R(V_k^* V_j)\,.
$$
\ENSURE An $h$-dimensional left dominant subspace $\cS'$ of $A$.
\STATE  Compute principal vectors $\{w_1,\ldots,w_j\}\subset \cW$ corresponding to $(\cW,\,\cH',j)$. 
\STATE  Set $\cW'=\text{Span}\{w_1,\ldots,w_j\}\subset \cW$ and compute $\cT''=\cT\ominus \phi(\Sigma_k)\cW'$.
\STATE  Compute principal vectors $\{z_1,\ldots,z_{h-j}\}\subset (\cH')^\perp$ corresponding to $(\cT'',\,(\cH')^\perp,h-j)$. 
\STATE Set $\cH''=\text{Span}\{z_1,\ldots,z_{h-j}\}$ and $\cS':=U_k(\cH'\oplus \cH'')=\cU_j\oplus U_k(\cH'')$.
\STATE Return: the left dominant subspace $\cS'\subset \K^m$. 
\end{algorithmic}
\end{algorithm}

\subsection{Proofs of Theorems \ref{pro 3.5}, \ref{third main result S1} and  \ref{coro para comp abc1}}\label{sec 4.3}

\begin{proof}[Proof of Theorem \ref{pro 3.5}] Under our present assumptions we have that 
$\Theta(R(X),\cV_j)<\pi/2\,I_j$ i.e. $R(V_j^*X)=\cV_j$. In this case we can argue as in the proof
of \cite[Proof of Theorem 2.1]{D19} and conclude that 
$$
\|\sin\Theta(R(\Psi_q),\cV_j)\|_{2,F}\leq \|\psi(\Sigma_{j,\perp})\cdot\Sigma_{j,\perp} ^2\|_2\,
\|(\psi(\Sigma_{j})\cdot\Sigma_j^2)^{-1}\|_2\, 
 \| V_{j,\perp}^* X( V_j^*  X)^\dagger\|_{2,F}\,.
$$
where we used that $\psi(\sigma_1),\ldots,\psi(\sigma_j)>0$ so that in this case $\psi(\Sigma_{j})\cdot\Sigma_j^2$ is an invertible matrix. Moreover, 
$$
\|\psi(\Sigma_{j,\perp})\cdot\Sigma_{j,\perp} ^2\|_2\,
\|(\psi(\Sigma_{j})\cdot\Sigma_j^2)^{-1}\|_2\, \leq \|\psi(\Sigma_{j,\perp})\|_2\,
\|\psi(\Sigma_{j})^{-1}\|_2\, \left(\frac{\sigma_{j+1}}{\sigma_j}\right)^2\,.
$$
The result now follows from the previous inequalities and the identity $\sigma_{j+1}/\sigma_j=(1+\gamma_j)^{-1}$.
\end{proof}

In what follows we make use of the next result from \cite{D19}.
\begin{lem}[\cite{D19}]\label{DIKM-I theorem 3}  \rm 
Assume that $k<\text{rank}( A)$, so that $\sigma_k>\sigma_{k+1}>0$, and let 
$$\gamma_k=\frac{\sigma_k-\sigma_{k+1}}{\sigma_{k+1}}>0\,.$$
Let $T_{2q}(x)\in\R[x]$ be the Chebyshev polynomial of the first kind of degree $2q$.
Then $T_{2q}(x)$ has even powers only. Furthermore, if we set $\psi(x)=T_{2q}(x/\sigma_{k+1})$ then
\begin{eqnarray*}\psi(\sigma_1)>\ldots>\psi(\sigma_{k+1})=1  \quad ,\quad  
\psi(\sigma_i)\geq \frac{1}{4}\,(1+\gamma_{k})\,{2^{2q\, \min\{\sqrt{\gamma_k}\coma 1\}}}
 \ &,& \ \peso{for} 1\leq i\leq k\,,\end{eqnarray*}
and $|\psi(\sigma_i)|\leq 1$, for $i\geq k+1$.
Hence, 
$$
\|\psi(\Sigma_k)^{-1}\|_2\, \|\psi(\Sigma_{k,\perp})\|_2\leq 
4\,\frac{2^{-2q \, \min\{\sqrt{\gamma_k}\coma 1\}}}{1+\gamma_k} \,.
$$
\qed
\end{lem}
We point out that the inequalities $\psi(\sigma_1)\geq \ldots\geq \psi(\sigma_{k+1})$ in the lemma above are a consequence of the super-linear growth for large input values (i.e., in this case for $x\geq \sigma_{k+1}$) of the gap amplifying Chebyshev polynomials (see \cite{D19}).

\begin{proof}[Proof of Theorem \ref{third main result S1}]
Let $T_{2q}(x)\in\R[x]$ be the Chebyshev polynomial of the first kind of degree $2q$.
Since $\sigma_{k+1}>0$ we can set $\psi(x)=T_{2q}(x/\sigma_{k+1})$. By Lemma \ref{DIKM-I theorem 3}
we get that 
$\psi(\sigma_1)>\ldots>\psi(\sigma_{k+1})=1$ and hence
$$ \| \psi(\Sigma_j)^{-1}\|_2=\frac{1}{\psi(\sigma_j)} \py \|\psi(\Sigma_{j+1})\|_2=\psi(\sigma_{j+1})=\psi(\sigma_k)\,.$$
Thus, Eq. \eqref{suert1} follows from Theorem \ref{pro 3.5} and the identities above.

\pausa

Consider $\Psi_q=V(\psi(\Sigma)\cdot \Sigma^2)V^*X$; then
$$
V_{k,\perp}^* \Psi_q (V_k^*\Psi_q)^\dagger= (\psi(\Sigma_{k,\perp})\cdot \Sigma_{k,\perp}^2)V_{k,\perp}^*X
((\psi(\Sigma_k)\cdot \Sigma_k^2) V_k^*X)^\dagger\,.
$$
Since $\psi(\Sigma_k)\cdot \Sigma_k^2\in\K^{k\times k}$ is invertible, then Proposition \ref{pro app 1} implies that 
$$((\psi(\Sigma_k)\cdot \Sigma_k^2) V_k^*X)^\dagger=(V_k^*X)^\dagger((\psi(\Sigma_k)\cdot \Sigma_k^2) P_{R(V_k^*X)})^\dagger\,.$$
Hence,  the previous identities imply that 
$$
\|V_{k,\perp}^* \Psi_q (V_k^*\Psi_q)^\dagger\|_{2,F}\leq 
\|\psi(\Sigma_{k,\perp})\cdot \Sigma_{k,\perp}^2\|_2 \, \|V_{k,\perp}^*X (V_k^*X)^\dagger\|_{2,F}
\|(\psi(\Sigma_k)\cdot \Sigma_k^2)^{-1}\|_2\,.
$$
Again, by Lemma \ref{DIKM-I theorem 3}, we get that
$\|\psi(\Sigma_{k,\perp})\cdot \Sigma_{k,\perp}^2\|_2\leq \sigma_{k+1}^2$
while 
$\|(\psi(\Sigma_k)\cdot \Sigma_k^2)^{-1}\|_2\leq (\psi(\sigma_k)\,\sigma_k^2)^{-1}$.
Putting everything together, we get that 
$$
\|V_{k,\perp}^* \Psi_q (V_k^*\Psi_q)^\dagger\|_{2,F}\leq \frac{1}{\psi(\sigma_k)}
\, \|V_{k,\perp}^*X (V_k^*X)^\dagger\|_{2,F}\, \left(\frac{\sigma_{k+1}}{\sigma_k}\right)^2\,.
$$
Thus, Eq. \eqref{suert2} now follows from the previous inequalities and the identity $\sigma_{k+1}/\sigma_k=(1+\gamma_k)^{-1}$.

Finally, notice that the first inequality in Eq. \eqref{suert3} was shown in Remark \ref{rem still true}, while the second inequality is a consequence of Lemma \ref{DIKM-I theorem 3}.
\end{proof}
Recall that given an $h$-compatible pair $(A,X)$ we defined $C(A,X,j,k)_{2,F}=C_{2,F}$ as follows: if $h\leq k<\text{rank}(A)$ and $1\leq j$ then
\beq \label{defi constC2}
C_{2,F}(V,X)=C_{2,F}=\max\left \{ 2\,\| V_{j,\perp}^* X( V_j^*  X)^\dagger\|_{2,F}\coma  \| V_{k,\perp}^* X( V_k^*  X)^\dagger\|_{2,F}\right\}\,.
\eeq If $j=0$ we let $C_{2,F}= \| V_{k,\perp}^* X( V_k^*  X)^\dagger\|_{2,F}$; if $k=\text{rank}(A)$ then we set $C_{2,F}=2\,\| V_{j,\perp}^* X( V_j^*  X)^\dagger\|_{2,F}$.

\begin{proof}[Proof of Theorem \ref{coro para comp abc1}]
Assume first that $1\leq j<h\leq k<\text{rank}(A)$; notice that $\sigma_h=\sigma_k >\sigma_{k+1}>0$, since $k<\text{rank}( A)$. To obtain the claimed convergence analysis we follow the strategy described in Remark \ref{rem strat}. 
Let $T_{2q}(x)$ be the Chebyshev polynomial of the first kind of 
degree $2q$ and set $\psi(x)=T_{2q}(x/\sigma_{k+1})$. 
By Theorem \ref{third main result S1}, 
if we let $\Psi_q=V(\psi(\Sigma)\cdot \Sigma^2)V^*X$ then 
\begin{equation}\label{suert12}
\|\sin\Theta(R(\Psi_q),\cV_j)\|_{2,F}\leq \frac{\psi(\sigma_{k})}{\psi(\sigma_{j})}\, 
 \| V_{j,\perp}^* X( V_j^*  X)^\dagger\|_{2,F}\,
\frac{1}{(1+\gamma_j)^2}
\end{equation}
\begin{equation}\label{suert22}
\| V_{k,\perp}^* \Psi_q (V_k^* \Psi_q)^\dagger\|_{2,F}\leq 
\frac{1}{\psi(\sigma_{k})}\, 
 \| V_{k,\perp}^* X( V_k^*  X)^\dagger\|_{2,F}\, \frac{1}{(1+\gamma_k)^2}\,.
\end{equation}
Since $V_h^*\Psi_q=\psi(\Sigma_h)\,\Sigma_h^2\,V_h^*X$ then $R(V_h^* \Psi_q)=R(V_h^*)$ so that the pair $(A,\Psi_q)$ is $h$-compatible and $\dim \Psi_q\geq h$. If we apply Theorem \ref{first main result} to this last pair with $t=0$ and $\phi(x)=x$ and consider Eqs. \eqref{suert12} and \eqref{suert22} above, we conclude that there exists
an $h$-dimensional left dominant subspace $\cS'\subset \K^m$ for $A$ such that 
\begin{eqnarray*}
\|\sin \Theta(\cK_0(A,\Psi_q),\cS')\|_{2,F}&\leq & 2\, \frac{\psi(\sigma_{k})}{\psi(\sigma_{j})}\, 
 \| V_{j,\perp}^* X( V_j^*  X)^\dagger\|_{2,F} \, \frac{1}{(1+\gamma_j)^2}+ \\ \\ & &
\frac{1}{1+\gamma_k} \,
\frac{1}{\psi(\sigma_{k})}\, 
 \| V_{k,\perp}^* X( V_k^*  X)^\dagger\|_{2,F}\, \frac{1}{(1+\gamma_k)^2}\,.
\end{eqnarray*} where we have used the identity $\sigma_{k+1}/\sigma_k=(1+\gamma_k)^{-1}$. 
Since $x^2\,\psi(x)\in \R[x]$ is a degree $2(q+1)$ polynomial with even powers only, then we get that 
$R(\Psi_q)\subset \R(\tilde K_q)$ where
$$
\tilde K_q:= K_q(A^*,AX)=(A^*(AX)\quad (A^*A) A^*(AX) \quad \cdots \quad (A^*A)^q A^*(AX))\in\mathbb K^{n\times (q+1)r}\,.
$$
Thus, $\cK_0(A,\Psi_q)\subset \cK_0(A,\tilde K_q)\subset \cK_{q+1}(A,X)$; using the properties of principal angles and the estimates above, we now see that 
\begin{eqnarray*} 
\|\sin \Theta(\cK_{q+1}(A,X) , \cS' )\|_{2,F}&\leq &\|\sin \Theta(\cK_0(A,\Psi_q),\cS')\|_{2,F} \\
&\leq & C_{2,F} \, \left( \frac{\psi(\sigma_k)}{\psi(\sigma_j)}\,\frac{1}{(1+\gamma_j)^2}
 +  \frac{1}{\psi(\sigma_{k})}\, 
 \frac{1}{(1+\gamma_k)^3}\right)\,,
\end{eqnarray*}
where $C_{2,F}=C_{2,F}(A,X)$ is defined as in Eq. \eqref{defi constC2}.

The case $j=0$ can be treated with similar arguments (the details are left to the reader).
\end{proof}

\begin{rem}[Construction of $\cS'$ as in Theorem \ref{coro para comp abc1}]\label{rem constr teo 37}
A simple inspection of the proof above shows that the $h$-dimensional left dominant subspace $\cS'\subset \K^m$ as in Theorem \ref{coro para comp abc1} can be constructed using Algorithm \ref{algoalgo2} with starting guess 
$\Psi_q=V(\psi(\Sigma)\cdot \Sigma^2) V^*X \in \K^{n\times r}$ and setting $t=0$ (the rest of the parameters remain the same). Notice that this is in accordance with the general strategy described in Remark \ref{rem strat}. 
\end{rem}

\subsection{Proof of Theorems \ref{teo nuevo bb1}, \ref{teo nuevo bb2}, \ref{other main result}
and \ref{teo supermain comple1}.
}\label{sec4.4}

Throughout this section we consider Notation \ref{nota21} and the notation from Theorem \ref{other main result}. 
In particular, we consider Algorithm \ref{algoalgo} with input: $ A\in\K^{m\times n}$, starting guess $ X\in \K^{n\times r}$; moreover, we set our target rank to: $1\leq h\leq $ rank $( A)$.
We let $U_{K}\in\K^{m\times d}$ denote the matrix whose columns form an orthonormal basis of the Krylov space $\cK_{q+1}$ constructed in terms of $ A$ and $ X$, for some fixed $q\geq 0$; further, we let $\hat{ U}_i\in \K^{m\times i}$ denote the matrix whose columns are the top $i$ columns of the output $\hat U_h$ of Algorithm \ref{algoalgo}, for $1\leq i\leq h$. 
We mostly focus on the case in which $1\leq j<h\leq k<\text{rank}(A)$; notice that in this case we have that $\sigma_h=\sigma_k >0$, since $h\leq \text{rank}( A)$. The cases $j=0$ or $k=\text{rank}( A)$ can be treated with similar arguments (the details are left to the reader). 

In what follows we make use of the constants $\delta_{2,F}$ given by: 
for $k<\text{rank}(A)$ then let $T_{2q}(x)$ be the Chebyshev polynomial of the first kind of degree $2q$, $\psi(x)=T_{2q}(x/\sigma_{k+1})$ and set
\begin{equation}\label{eq rec delt}
\delta_{2,F}:=\sqrt 2\, 
C_{2,F} \, \left( \frac{\psi(\sigma_k)}{\psi(\sigma_j)}\,\frac{1}{(1+\gamma_j)}
 +  \frac{1}{\psi(\sigma_{k})}\, 
 \frac{1}{(1+\gamma_k)^2}\right)\,,
\end{equation} where $C_{2,F}$ is defined as in Eq. \eqref{defi constC2}.
In case $j=0$ the first term should be omitted.
In case $k=\text{rank}(A)$ and $1\leq j$, then let $\psi(x)=T_{2q}(x/\sigma_{k})$ and set 
\begin{equation}\label{eq rec delt2}
\delta_{2,F}:= \frac{\sqrt 2}{\psi(\sigma_j)}\,\frac{C_{2,F}}{(1+\gamma_j)}
\,.
\end{equation}

\begin{proof}[Proof of Theorem \ref{teo nuevo bb1}]
We first assume that $1\leq j<h\leq k<\text{rank}(A)$ so $\gamma_k>0$.
Let $T_{2q}(x)$ be the Chebyshev polynomial of the first kind of 
degree $2q$ and set $\psi(x)=T_{2q}(x/\sigma_{k+1})$. 
Consider $\tilde \Psi_q:=U \,(\psi(\Sigma)\cdot \Sigma ) \,V^*X\in \K^{m\times r}$ and notice that 
$A^*\tilde \Psi_q= V(\,(\psi(\Sigma)\cdot \Sigma^2 ) \,V^*X=\Psi_q$.
Arguing as in the proof of Theorem \ref{third main result S1} we get that 
$$
\|\sin\Theta(R(\tilde \Psi_q),\cU_j)\|_{2,F}\leq \frac{\psi(\sigma_{k})}{\psi(\sigma_{j})}\, 
 \| V_{j,\perp}^* X( V_j^*  X)^\dagger\|_{2,F}\,
\frac{1}{(1+\gamma_j)}
\, , 
$$
$$
\|U_{k,\perp}^* \tilde \Psi_q (U_k^*\tilde \Psi_q)^\dagger\|_{2,F}\leq \frac{1}{\psi(\sigma_k)}
\, \|V_{k,\perp}^*X (V_k^*X)^\dagger\|_{2,F}\, \frac{1}{(1+\gamma_k)}\,.
$$
Notice that $U_h^*\tilde \Psi_q=\psi(\Sigma_h)\,\Sigma_h^2 \,V_h^*X$ so that $R(U_h^*\tilde \Psi_q)=R(U_h^*)$ and the pair $(A^*,\tilde \Psi_q)$ is $h$-compatible. Using the previous estimates together with Theorem \ref{first main result} applied to the matrix $A^*$, the starting guess matrix $\tilde \Psi_q$ and $t=0$ (with $\phi(x)=x$) we  get that there exists an $h$-dimensional right dominant subspace $\tilde \cS\subset \mathbb K^n$ of $A$ such that 
\begin{equation}\label{eq agreg1}
\| \sin \Theta(R(\Psi_q),\tilde \cS)\|_{2,F}\leq 
C_{2,F} \, \left( \frac{\psi(\sigma_k)}{\psi(\sigma_j)}\,\frac{1}{(1+\gamma_j)}
 +  \frac{1}{\psi(\sigma_{k})}\, 
 \frac{1}{(1+\gamma_k)^2}\right)=\frac{\delta_{2,F}}{\sqrt{2}}\,,
\end{equation}
 where we have used that $R(A^*\tilde \Psi_q)=R(\Psi_q)$ and $C_{2,F}$ is defined as in Eq. \eqref{defi constC2}.
Hence, by Eq. \eqref{eq agreg1} and our hypotheses we now see that
\beq \label{eq es menor que pisobrecuatro}
 \Theta(R(\Psi_{q}),\tilde \cS)\leq \frac{\pi}{4}\, I\,.
\eeq
 On the other hand, arguing as in the proof of Theorem \ref{coro para comp abc1} we see that $ R(A \Psi_{q})\subset \cK_{q+1}(A,X)$. We now set $A_{\tilde \cS}=A\,P_{\tilde \cS}$, where $P_{\tilde \cS}$ denotes the orthogonal projection onto 
$\tilde \cS\subset \K^n$. In this case we get that $A=A_{\tilde \cS}+(A-A_{\tilde \cS})$ so that 
$A_{\tilde \cS}(A-A_{\tilde \cS})^*=AP_{\tilde \cS}(I-P_{\tilde \cS})A^*=0$. If we let 
$V_{\tilde \cS}\in\K^{n\times h} $ denote the top right singular vectors of $A_{\tilde \cS}$ then
$R(V_{\tilde \cS})=\tilde \cS$. Notice that Eq. \eqref{eq es menor que pisobrecuatro} implies that
$R(V_{\tilde \cS}^*\Psi_{q})=R(V_{\tilde \cS}^*)$ so we can apply Lemma \ref{lem C1} in this context. Hence,
we consider the principal vectors $\{w_1,\ldots,w_h\}\subset \Psi_{q}$, corresponding to the pair of subspaces 
$(R(\Psi_{q}),\tilde \cS)$; we also let $Q\in\K^{n\times h}$ be an isometry with columns $w_1,\ldots,w_h$ so that
$R(AQ)\subset R(A \Psi_{q})\subset  \cK_{q+1}$. Recall that $U_K$ (being the output of Algorithm \ref{algoalgo}
with power parameter set to $q+1$)
denotes the matrix whose columns form an orthonormal basis of the Krylov space $\cK_{q+1}$. Then, the previous facts together with Lemma \ref{lem C1} show that 
\begin{eqnarray*}
\|A_{\tilde \cS}-U_KU_K^* A_{\tilde \cS}\|_{2,F}&\leq& \|(I-AQ(AQ)^\dagger)A_{\tilde \cS}\|_{2,F}= \|A_{\tilde \cS} -AQ(AQ)^\dagger A_{\tilde \cS}\|_{2,F}
\\ &\leq & \|A-A_{\tilde \cS}\|_2\, \|\tan \Theta(R(\Psi_{q}),\tilde \cS)\|_{2,F}\,.
\end{eqnarray*}
On the one hand, $\|A-A_{\tilde \cS}\|_2=\sigma_{h+1}$, since $\tilde \cS$ is an $h$-dimensional right dominant subspace of $A$. On the other hand, Eq. \eqref{eq es menor que pisobrecuatro} also implies that
$$
 \|A-A_{\tilde \cS}\|_2\, \|\tan \Theta(R(\Psi_{q}),\tilde \cS)\|_{2,F}\leq \sigma_{h+1}\, \sqrt 2 \,\|
\sin \Theta(R(\Psi_{q}),\tilde \cS)\|_{2,F}\leq \sigma_{h+1}\,\delta_{2,F}\,,
$$ where we have also used Eq. \eqref{eq agreg1}.
\end{proof}

\begin{proof}[Proof of Theorem \ref{teo nuevo bb2}]

\medskip

\noindent We keep using the notation from the proof of Theorem \ref{teo nuevo bb1} above. As a consequence of Theorem \ref{teo nuevo bb1} and Lidskii's inequality for singular values, we get that 
$$
|\sigma_i(A\,P_{\tilde \cS})- \sigma_i(U_KU_K^*A \,P_{\tilde \cS} )|\leq \|
A\,P_{\tilde \cS}- U_KU_K^*A \,P_{\tilde \cS} \|_2\leq \sigma_{h+1}\,\delta_2
\coma 1\leq i\leq h\,.$$ 
Therefore, for $1\leq i\leq h$, we have that 
$$
\sigma_i(A)\geq \sigma_i(U_KU_K^*A )=\sigma_i(\hat A_h)\geq \sigma_i(U_KU_K^*AP_{\tilde \cS})\geq \sigma_i(AP_{\tilde \cS} )-\sigma_{h+1}\,\delta_2\,.
$$ The result now follows from the previous facts and the identities $\sigma_i(AP_{\tilde \cS} )=\sigma_i(A)$, for $1\leq i\leq h$.
\end{proof}

\begin{proof}[Proof of Theorem \ref{other main result}]

\medskip

\noindent  We keep using the notation from the proof of Theorem \ref{teo nuevo bb1} above. 
In particular, we consider the existence of an $h$-dimensional right dominant subspace $\tilde \cS$ for $ A$ that satisfies
Eq. \eqref{eq agreg1}. We now argue as in the proof of \cite[Theorem 2.3.]{D19} and consider the estimate for the Frobenius norm first. Indeed, by \cite[Lemma 8]{BDMM14} we have that 
$$
 A-\hat{ U}_i\hat{ U}_i^*  A= A- U_{K}( U_{K}^*  A)_i\peso{for} 1\leq i\leq h\,,
$$ where $( U_{K}^*  A)_i$ denotes a best rank-$i$ approximation of $ U_{K}^*  A$. By the same result, we also get that $ U_{K}( U_{K}^*  A)_i$ is a best rank-$i$ approximation of $ A$ from $R(U_K)$ in the Frobenius norm i.e.,
\begin{equation}\label{eqsec21}
\| A-  U_{K}( U_{K}^*  A)_i\|_F=\min_{\text{rank}( Y)\leq i}\|
 A-  U_{K} Y\|_F\,.
\end{equation}
We now consider a SVD, $ A= U \Sigma V^*$ such that 
the top $h$ columns of $ V$ span the $h$-dimensional right dominant 
subspace $R( V_h)=\cV_h=\tilde \cS$ 
 (recall that this can always be done). For $1\leq i\leq h$ we set 
$$
 A= A_i+ A_{i,\perp}\peso{where} A_i= U_i  \Sigma_i  V_i^* \quad \text{and}\quad  A_{i,\perp}= A- A_i\,.
$$ Then, by \cite[Lemma 7.2]{D19} we get that 
\begin{equation}\label{eqsec22}
\| A-\hat{ U}_i\hat{ U}_i^* A\|_F^2\leq \| A- A_i\|_F^2+ \|
 A_i-{ U_{K}} U_{K}^* A_i\|_F^2\,.
\end{equation}
Now we bound the second term in Eq. \eqref{eqsec22}. 
Under the present notation, we get that $ \Theta (R( V_h),R(\Psi_{q}))\leq \frac{\pi}{4} \, I_h$, by Eq. \eqref{eq es menor que pisobrecuatro}. Hence, 
$ \Theta (R( V_i), R(\Psi_{q}))\leq \frac{\pi}{4} \, I_i$ (see Section \ref{sec aux angles})
and we see that $ R(V_i^* \Psi_{q})=R( V_i^*)$, for $1\leq i\leq h$. Thus, we can apply Lemma \ref{lem C1} in this context. Hence,
we consider the principal vectors $\{ w_1,\ldots, w_i\}\subset R(\Psi_{q})$ corresponding to the pair
$(R(\Psi_{q}),R( V_i))$. Moreover, we let $ Q\in\K^{n\times i}$ be an isometry with
$R( Q)=\text{Span}\{ w_1,\ldots, w_i\}$ so that $R( A Q)\subset  R(A\Psi_{q})\subset \cK_{q+1}$. 
The previous facts together with Lemma \ref{lem C1} show that 
\begin{eqnarray*}
\| A_i-{ U_{K}} U_{K}^* A_i\|_F&\leq & \|( I- A  Q ( A  Q)^\dagger) A_i\|_F
=\| A_i - A  Q ( A  Q)^\dagger A_i\|_F
\\ &\leq &\| A- A_i\|_2\,\|\tan \Theta (R(\Psi_{q}),R( V_i))\|_F \\ &\leq &
\| A- A_i\|_2\,\|\tan \Theta (R(\Psi_{q}),R( V_h))\|_F\,.
\end{eqnarray*}
By Eq. \eqref{eq es menor que pisobrecuatro} we get that
$ \Theta (R( V_h),R(\Psi_{q}))=\Theta (\tilde \cS,R(\Psi_{q}))\leq \frac{\pi}{4}\,  I$; then,
$$
\|\tan \Theta (R(\Psi_{q}),R( V_h))\|_{F}\leq 
\sqrt 2\, \|\sin \Theta (R(\Psi_{q}),R( V_h))\|_{F} \leq \delta_F
$$
where we have used 
Eq. \eqref{eq agreg1}.
Therefore, the previous inequalities imply that
\begin{equation} \label{eq prueba part1}
\| A-\hat{ U}_i\hat{ U}_i^* A\|_F^2\leq \| A- A_i\|_F^2+ (\sigma_{i+1} \cdot\delta_F)^2\,.
\end{equation} 
 This proves the upper bound in Eq. \eqref{eq conc teo aprox} for the Frobenius norm. To prove the bound for the spectral norm, recall that by \cite[Theorem 3.4.]{Gu15} we get that Eq. \eqref{eq prueba part1} implies that $$
\| A-\hat{ U}_i\hat{ U}_i^* A\|_2^2\leq \| A- A_i\|_2^2+ (\sigma_{i+1} \cdot\delta_F)^2\,,$$
since $\text{rank}(\hat{ U_i}\hat{ U}_i^* A)\leq i$. The upper bound in \eqref{eq conc teo aprox} for the spectral  norm follows from this last fact.
\end{proof}

\begin{proof}[Proof of Theorem \ref{teo supermain comple1}]\label{rem para coro1212}
Let us fix an $h$-compatible pair $(A,X)$ and assume that $0\leq j<h\leq k< \text{rank}(A)$. 
We follow the conventions: $\gamma_0=+\infty$ in case $j=0$.  Let $T_{2q}(x)$ be the Chebyshev polynomial of the first kind of degree $2q$ and set $\psi(x)=T_{2q}(x/\sigma_{k+1})$; in this case, by Theorem \ref{third main result S1} 
we get that 
$$
\frac{\psi(\sigma_{k})}{\psi(\sigma_{j})}\leq 
\frac{1}{(1+\gamma_j)^{2q}} \py
\frac{1}{\psi(\sigma_{k})}\leq 4\, \frac{2^{-2q\,\min\{\sqrt{\gamma_k}\, , \, 1 \}}}{(1+\gamma_k)}\,.
$$ Hence, if we let $\delta_{2,F}$ be defined as in Eq. 
\eqref{eq rec delt} then we get that 
\begin{equation}\label{eq delta vs conc1}
\delta_{2,F}\leq \sqrt{2}\, C_{2,F}\,\left( \frac{1}{(1+\gamma_j)^{2q+1}}+ 4\,\frac{2^{-2q\,\min\{\sqrt{\gamma_k}\, , \, 1 \}}}{(1+\gamma_k)^3}\right)\,.
\end{equation}
Thus, if Eq. \eqref{eq cond conc1} is satisfied then $\delta_2\leq 1$. In this case, we can apply Theorems 
\ref{teo nuevo bb2} and \ref{other main result} and get, for $1\leq i\leq h$:
$$
 \sigma_{i}-\sigma_{h+1}\cdot\delta_2\leq \sigma_i(\hat A_h) \leq \sigma_i \py
\| A-\hat{ U}_i\hat{ U}_i^* A\|_{2,F}\leq \| A- A_i\|_{2,F} + \sigma_{i+1} \cdot \delta_F\,.
$$
Therefore, the result is a consequence of the previous two bounds together with the estimate in 
Eq. \eqref{eq delta vs conc1}.
\end{proof}

\section{Appendix}\label{apendixity}

In this section we include several technical results that were used in the proofs of the main results.
Most of these technical results are elementary and can be found in the literature; we include the versions that are best  suited for our exposition together with their proofs, for the convenience of the reader.

\begin{pro}\label{pro app ang entre subs}
Let $ V\in\K^{n\times k}$ be an isometry and let $\cV',\,\cX\subset \K^n$ be subspaces such that 
$\dim\cX\geq \dim  \cV'=j$, $\cV'\subset (\ker  V^*)^\perp=R( V)$ and $ \Theta(\cX,\cV')< \pi/2\, I_j$. Then $\cW= V^*\cX\subset\K^k$ is such that $\dim\cW\geq j$ and if we let $\cH'= V^*\cV'$ then 
$$
 \Theta(\cW,\cH')\leq  \Theta(\cX,\cV')\in\R^{j\times j}\,.
$$
\end{pro}
\begin{proof}
First notice that 
$$
 P_\cX  P_{\cV'}  P_\cX\leq 
 P_\cX  V V^* P_\cX\,.
$$ By hypothesis rank$( P_\cX  P_{\cV'}  P_\cX)=j$ which shows that $\dim \cW=\text{rank}( V^* P_\cX)\geq j$. On the other hand, since $ V$ is an isometry then 
$ \Theta(\cW,\cH')= \Theta( V \cW, V \cH')=
 \Theta( V V^*\cX,{\cV'})$.
Consider $ D= V  V^*  P_\cX  V V^* $; then
$R( D)= V V^*\cX$, so $\dim R( D)=\dim\cW\geq j$. Moreover,
$$0\leq  D\leq  P_{R( D)}\implies  P_{\cV'}
 P_{\cX}  P_{\cV'}= P_{\cV'}
 D  P_{\cV'}\leq 
 P_{\cV'}
 P_{R( D)}  P_{\cV'}\,,$$where we used that $ P_{\cV'}  V V^*= P_{\cV'}$.
Then, $ \cos^2  \Theta(\cX,{\cV'})\leq  \cos^2  \Theta( V V^*\cX,{\cV'})\in\R^{j\times j}$ and the result follows from the fact that $f(x)=\cos^2(x)$ is a decreasing function on $[0,\pi/2]$.
\end{proof}

\begin{pro}\label{pro app 2}
Let $ B\in\K^{k\times k}$ be such that $ B= B^*$ and let $\cH',\,\cW'\subset \K^k$ be subspaces such that $ P_{\cH'}  B =  B  P_{\cH'}$, 
$\dim\cH'=\dim\cW'$  
and $\cH',\,\cW'\subset \ker( B)^\perp$. If we let $ B \cW'=\cT'$, 
$$
\|\sin  \Theta(\cH',\cT')\|_{2,F}\leq 
\| B( I- P_{\cH'})\|_2 \,\| B^\dagger\|_2 \|\sin \Theta(\cH',\cW')\|_{2,F}\,.$$
\end{pro}

\begin{proof}
Notice that $( B P_{\cH'})^\dagger =  P_{\cH'}  B^\dagger =  B^\dagger  P_{\cH'}$. Then,
$$( I- P_{\cH'}) P_{\cT'}=( I- P_{\cH'}) ( B P_{\cW'}) \, ( B P_{\cW'})^\dagger=
( B ( I- P_{\cH'}))\, ( I- P_{\cH'})  P_{\cW'} \, ( B P_{\cW'})^\dagger\,.
$$ 
Also, notice that $( B P_{\cW'})^\dagger =   P_{\cW'}  B^\dagger  P_{\cT'}$; in particular, $\|( B P_{\cW'})^\dagger \|_2\leq \| B^\dagger\|_2$. Finally, since $\dim\cT'=\dim\cW'=\dim\cH'$ the previous facts imply that
$$
\|\sin \Theta(\cH',\cT')\|_{2,F}\leq 
\| B( I- P_{\cH'})\|_2 \,\| B^\dagger\|_2 
\|\sin \Theta(\cH',\cW')\|_{2,F}\,.$$
\end{proof}

\begin{pro}\label{prop pegoteo de subespacios}
Let $\cT',\,\cT''$ and $\cH',\,\cH''$ be pairs of mutually orthogonal subspaces in $\K^k$, such that $\dim(\cH')\leq\dim(\cT')$ and $\dim(\cH'')\leq \dim(\cT'')$.
 Consider the subspaces in $\K^k$ given by the (orthogonal) sums 
$\cT=\cT'\oplus \cT''$ and $\cH=\cH'\oplus \cH''$, so $\dim (\cH)\leq \dim(\cT)$. In this case we have that 
$$
\|\sin \Theta(\cT,\cH)\|_{2,F}\leq \|\sin \Theta(\cT',\cH')\|_{2,F}+\|\sin \Theta(\cT'',\cH'')\|_{2,F}\,.
$$
\end{pro}

\begin{proof}
We will make use of the following fact: given two subspaces $\cR,\,\cS\subset \mathbb K^k$ of the same dimension $d$ then the non-zero singular values of $P_\cR-P_\cS$ coincide with the non-zero entries of $(\sin(\theta_d),\sin(\theta_d),\ldots,\sin(\theta_1),\sin(\theta_1))$, where $\Theta(\cR,\cS)=\text{diag}(\theta_1,\ldots,\theta_d)$. In particular, $\|\sin \Theta(\cR,\cS)\|_{2}=\|P_\cR-P_\cS\|_2$ and 
$\sqrt{2}\, \|\sin \Theta(\cR,\cS)\|_{F}=\|P_\cR-P_\cS\|_F$.

Let $\cT'_0\subset \cT'$ and $\cT''_0\subset \cT''$ be such that $\dim\cT'_0=\dim\cH'$, $\dim\cT''_0=\dim\cH''$ and $\Theta(\cT',\cH')=\Theta(\cT'_0,\cH')$, $\Theta(\cT'',\cH'')=\Theta(\cT''_0,\cH'')$.
Let $\cT_0=\cT'_0\oplus \cT''_0\subset \cT$; since $\dim \cH=\dim\cT_0$ we get that 
$\Theta(\cH,\cT)\leq \Theta (\cH,\cT_0)$. The previous remarks show that we can assume further that 
$\dim \cT'=\dim \cH'$ and $\dim \cT''=\dim \cH''$, so that $\dim \cT=\dim\cH$. In this case we have that
\begin{eqnarray*}
\|\sin \Theta(\cT,\cH)\|_2&=&\| P_\cT-P_\cH\|_2\leq \| P_{\cT'}-P_{\cH'}\|_2+
\| P_{\cT''}-P_{\cH''}\|_2\\ &=&\|\sin \Theta(\cT',\cH')\|_2+\|\sin \Theta(\cT'',\cH'')\|_2\,,
\end{eqnarray*}
where we used that $P_\cT=P_{\cT'}+P_{\cT''}$ and $P_\cH=P_{\cH'}+P_{\cH''}$. The Frobenius norm 
can be handled similarly.
\end{proof}

\begin{pro}\label{pro app 1}
Let $ B\in\K^{p\times q}$ and let $C\in\K^{q\times r}$ with 
$R( C)=\cV\subset\K^q$ such that 
$\cV\subset \ker  B^\perp$. Then $$( B C)^\dagger=
 C^\dagger ( B P_\cV)^\dagger\,.$$
\end{pro}
\begin{proof}
In this case $R( B C)= B\cV$ and $\ker  B C=\ker  C$. Moreover,
$$ B C  C^\dagger ( B P_\cV)^\dagger=
 B P_\cV( B P_\cV)^\dagger= P_{ B\cV}\py
$$ 
$$
  C^\dagger ( B P_\cV)^\dagger  B C  =
 C^\dagger ( B P_\cV)^\dagger  B P_\cV  C  
= C^\dagger  P_{\ker( B P_\cV)^\perp} C  = P_{\ker  C^\perp}\,,
$$ where we used that $\ker( B P_\cV)=\cV^\perp$, since $\cV\subset \ker B^\perp$. 
\end{proof}

\smallskip

Let $ C\in\K^{m\times c}$ have rank $p$. For $1\leq i\leq p$ we define 
$$\cP^{\xi}_{ C,\,i}( A)=  C\cdot \text{argmin}_{\text{rank}( Y)\leq i} \| A- C Y\|_\xi \peso{for} \xi=2,F\,.$$
Due to the optimality properties of the projection $ C C^\dagger$ (see \cite{Gu15}) we get that
\begin{equation} \label{eq opt de cc+ vs pc}
\| A- C C^\dagger A\|_{\xi}\leq \| A-\cP^{\xi}_{ C,\,i}( A)\|_{\xi} \peso{for} \xi=2,F\,.
\end{equation}
The following result is \cite[Lemma C.5]{WZZ15} (see also \cite{BDMM14}).
\begin{lem}[\cite{WZZ15}]\label{lem C5}
Let $ A\in \K^{m\times n}$  and consider a decomposition $ A= A_1+ A_2$, with $\text{rank}( A_1)=i$. Let $ V_1\in \K^{n\times i}$ denote the top right singular vectors of $ A_1$. Let $ Z\in\K^{n\times p}$ be such that 
$\text{rank}( V_1^* Z)=i$ and let $ C= A Z$. 
Then $\text{rank}( C)\geq i$ and
$$
\| A_1-\cP^{\xi}_{ C,\,i}( A_1)\|_{\xi}\leq \| A_2 Z( V_1^* Z)^\dagger\|_{\xi} \peso{for} \xi=2,F\,.
$$
\end{lem}
The following is a small variation of \cite[Lemma C.1]{WZZ15}
\begin{lem}\label{lem C1}
Let $ A\in \K^{m\times n}$  and consider the decomposition $ A= A_1+ A_2$, with $ A_1 A_2^*=0$ and $\text{rank}( A_1)=i$. Let $ V_1\in\K^{n\times i}$ and $ V_2\in\K^{n\times (n-i)}$ denote the top right singular vectors of $ A_1$ and $ A_2$ respectively. 
Let $\tilde \cK\subset \K^n$ be a subspace such that $ V_1^*(\tilde \cK)=R( V_1^*)$.
Let $\{ x_1,\ldots, x_i\}\subset \tilde \cK$ denote the principal vectors corresponding to the pair $(\tilde \cK,R( V_1))$ and let $ Q\in\K^{n\times i}$ be an isometry with $R( Q)=\text{Span}(\{ x_1,\ldots, x_i\})\subset \tilde \cK$. Then,
$$
\| A_1-( A Q)( A Q)^\dagger  A_1\|_{2,F}\leq 
\| A-  A_1\|_2\ \|\tan  \Theta(\tilde \cK,R( V_1))\|_{2,F}\,.
$$ 
\end{lem}
\begin{proof}
Notice that by construction
$$
 \Theta(R( Q),R( V_1))=  \Theta(\tilde \cK,R( V_1))<\frac{\pi}{2}\, I\,.
$$Then, we get that $\text{rank}( A Q)=i$. Hence, we have that 
\begin{eqnarray*}
\| A_1-( A Q)( A Q)^\dagger  A_1\|_{2,F}&\leq &
 \| A_1-\cP^{2,F}_{ A Q,\,i}( A_1)\|_{2,F}\leq 
\| A_2 Q( V_1^* Q)^\dagger\|_{2,F}\\ &\leq &
\| A_2\|_2\, \| V_2^* Q( V_1^* Q)^\dagger\|_{2,F}\\ &=&
\| A_2\|_2\ \|\tan  \Theta(R( Q),R( V_1))\|_{2,F}\\ &=&
\| A-  A_1\|_2\ \|\tan  \Theta(\tilde \cK,R( V_1))\|_{2,F}\,,
\end{eqnarray*}where  we have used Eq. \eqref{eq opt de cc+ vs pc}, 
Lemma \ref{lem C5}, that the isometry $ V_2$ satisfies that $ A_2= A_2  V_2 V_2^*$  and the identity
$\| V_2^* Q( V_1^* Q)^\dagger\|_{2,F}=\|\tan  \Theta(R( Q),R( V_1))\|_{2,F}
 V_1))\|_{2,F}$, that holds by \cite[Lemma 4.3]{D19}, since $\text{rank}( V_1^* Q)=i$.
\end{proof}

\smallskip

\noindent {\bf Acknowledgment}. We would like to thank the reviewers for several suggestions that helped improve the exposition of the results of the manuscript. Special thanks to Francisco Arrieta Zuccalli for several suggestions
and comments that lead to simplifications of the exposition of the results. This research wa s partially supported by  CONICET (PIP 2016 0525/ PIP 0152 CO), ANPCyT (2015 1505/ 2017 0883) and FCE-UNLP (11X974).
\end{document}